\documentclass[reqno]{amsart}%
\usepackage{amssymb,mathtools,anysize}
\marginsize{2.5cm}{2.5cm}{2.5cm}{2.5cm}
\usepackage{ifpdf}
\ifpdf
 \usepackage[hyperindex]{hyperref}%,pagebackref
\else
 \expandafter\ifx\csname dvipdfm\endcsname\relax
 \usepackage[hypertex,hyperindex]{hyperref}%
 \else
 \usepackage[dvipdfm,hyperindex]{hyperref}%
 \fi
\fi

\allowdisplaybreaks[4]
\numberwithin{equation}{section}
\theoremstyle{plain}
\newtheorem{thm}{Theorem}[section]
\newtheorem{cor}{Corollary}[section]
\theoremstyle{remark}
\newtheorem{rem}{Remark}[section]
\DeclareMathOperator{\td}{d}
\DeclareMathOperator{\arcsinh}{arcsinh}
\DeclareMathOperator{\arctanh}{arctanh}
\DeclareMathOperator{\arccosh}{arccosh}
\DeclareMathOperator{\arcsec}{arcsec}
\DeclareMathOperator{\bell}{B}
\DeclareMathOperator{\ls}{Ls}
\DeclareMathOperator{\te}{e}
\DeclareMathOperator{\ti}{i}

\begin{document}

\title[Series expansions of powers of inverse sine and tangent]
{Maclaurin's series expansions for positive integer powers of inverse (hyperbolic) sine and related functions, specific values of partial Bell polynomials, and two applications}

\author[B.-N. Guo]{Bai-Ni Guo}
\address{School of Mathematics and Informatics, Henan Polytechnic University, Jiaozuo 454010, Henan, China}
\email{bai.ni.guo@gmail.com, bai.ni.guo@hotmail.com}
\urladdr{\url{https://orcid.org/0000-0001-6156-2590}}

\author[D. Lim]{Dongkyu Lim*}
\address{Department of Mathematics Education, Andong National University, Andong 36729, Republic of Korea}
\email{\href{mailto: D. Lim <dgrim84@gmail.com>}{dgrim84@gmail.com}, \href{mailto: D. Lim <dklim@anu.ac.kr>}{dklim@anu.ac.kr}}
\urladdr{\url{https://orcid.org/0000-0002-0928-8480}}

\author[F. Qi]{Feng Qi*}
\address{School of Mathematical Sciences, Tiangong University, Tianjin 300387, China}
\email{\href{mailto: F. Qi <qifeng618@gmail.com>}{qifeng618@gmail.com}, \href{mailto: F. Qi <qifeng618@hotmail.com>}{qifeng618@hotmail.com}, \href{mailto: F. Qi <qifeng618@qq.com>}{qifeng618@qq.com}}
\urladdr{\url{https://qifeng618.wordpress.com}, \url{https://orcid.org/0000-0001-6239-2968}}

\dedicatory{Dedicated to people facing and battling COVID-19}

\begin{abstract}
In the paper, the authors establish Maclaurin's series expansions and series identities for positive integer powers of the inverse sine function, for positive integer powers of the inverse hyperbolic sine function, for the composite of incomplete gamma functions with the inverse hyperbolic sine function, for positive integer powers of the inverse tangent function, and for positive integer powers of the inverse hyperbolic tangent function, in terms of the first kind Stirling numbers and binomial coefficients, apply the newly established Maclaurin's series expansion for positive integer powers of the inverse sine function to derive a closed-form formula for specific values of partial Bell polynomials and to derive a series representation of the generalized logsine function, and deduce several combinatorial identities involving the first kind Stirling numbers. Some of these results simplify and unify some known ones. All of these newly established Maclaurin's series expansions of positive integer powers of the inverse (hyperbolic) sine and tangent functions can be used to derive infinite series representations of the circular constant Pi and of positive integer powers of Pi.
\end{abstract}



\thanks{*Corresponding author}

\thanks{This paper was typeset using\AmS-\LaTeX}

\maketitle
\tableofcontents

\section{Outlines}

Basing on conventions in community of mathematics, we use the notations
\begin{align*}
\mathbb{N}&=\{1,2,\dotsc\}, & \mathbb{N}_-&=\{-1,-2,\dotsc\}, & \mathbb{N}_0&=\{0,1,2,\dotsc\},\\
\mathbb{Z}&=\{0,\pm1,\pm2,\dotsc\}, & \mathbb{R}&=(-\infty,\infty), & \mathbb{C}&=\bigl\{x+\ti y: x,y\in\mathbb{R}, \ti=\sqrt{-1}\,\bigr\}.
\end{align*}
\par
In general, it is not much difficult to obtain Maclaurin's series expansions of powers of (fundamental) elementary functions and hypergeometric functions. The difficulty has been demonstrated in the papers~\cite{Bessel-Power-Polyn.tex, Moll-Vignat-IJNT-2014, AJOM-D-16-00138.tex}, for example. Many special cases of Maclaurin's series expansions of the positive integer power $(\arcsin t)^m$ for $m\in\mathbb{N}$ have been reviewed and surveyed in the paper~\cite{AIMS-Math20210491.tex}.
\par
In Section~\ref{sec-arcsin-series-expan} of this paper, we will discover Maclaurin's series expansions of the power functions $\bigl(\frac{\arcsin t}{t}\bigr)^m$ and $\frac{(\arcsin t)^{m}}{\sqrt{1-t^2}\,}$ for $m\in\mathbb{N}$. These series expansions simplify and unify previous results in~\cite[Section~2 and~5]{AIMS-Math20210491.tex}. In this section, we will also derive two combinatorial identities for finite sums involving the first kind Stirling numbers $s(n,k)$, which can be generated~\cite{Charalambides-book-2002, Comtet-Combinatorics-74} by
\begin{equation}\label{gen-funct-3}
\frac{[\ln(1+x)]^k}{k!}=\sum_{n=k}^\infty s(n,k)\frac{x^n}{n!},\quad |x|<1
\end{equation}
and satisfy diagonal recursive relations
\begin{equation*}%\label{Stirling1ID}
\frac{s(n+k,k)}{\binom{n+k}{k}}
=\sum_{\ell=0}^{n} (-1)^{\ell}\frac{\langle k\rangle_{\ell}}{\ell!} \sum_{m=0}^\ell(-1)^m\binom{\ell}{m}\frac{s(n+m,m)}{\binom{n+m}{m}}
\end{equation*}
and
\begin{align*}%\label{s(n-k)=s(n-k)-id}
s(n,k)&=(-1)^{k}\sum_{m=1}^{n}(-1)^{m}\sum_{\ell=k-m}^{k-1}(-1)^{\ell} \binom{n}{\ell}\binom{\ell}{k-m} s(n-\ell,k-\ell)\\ %\label{1stirling-diagonal-eq}
&=(-1)^{n-k}\sum_{\ell=0}^{k-1}(-1)^\ell\binom{n}{\ell} \binom{\ell-1}{k-n-1}s(n-\ell,k-\ell)
\end{align*}
in~\cite[p.~23, Theorem~1.1]{1st-Stirl-No-adjust.tex} and~\cite[p.~156, Theorem~4]{AADM-2821.tex}.
\par
In Section~\ref{sec-Bell-special-values}, applying Maclaurin's series expansion of $\bigl(\frac{\arcsin t}{t}\bigr)^m$ established in Section~\ref{sec-arcsin-series-expan}, we will present a closed-form formula of specific values
\begin{equation}\label{special-values-Oertel}
\bell_{2n,k}\biggl(0,\frac{1}{3},0,\frac{9}{5},0,\frac{225}{7},\dotsc, \frac{1+(-1)^{k+1}}{2}\frac{[(2n-k)!!]^2}{2n-k+2}\biggr)
\end{equation}
for $2n\ge k\in\mathbb{N}$, where partial Bell polynomials $\bell_{n,k}$ for $n\ge k\in\mathbb{N}_0$ are defined in~\cite[Definition~11.2]{Charalambides-book-2002} and~\cite[p.~134, Theorem~A]{Comtet-Combinatorics-74} by
\begin{equation}\label{Bell2nd-Dfn-Eq}
\bell_{n,k}(x_1,x_2,\dotsc,x_{n-k+1})=\sum_{\substack{1\le i\le n-k+1\\ \ell_i\in\mathbb{N}_0\\ \sum_{i=1}^{n-k+1}i\ell_i=n\\
\sum_{i=1}^{n-k+1}\ell_i=k}}\frac{n!}{\prod_{i=1}^{n-k+1}\ell_i!} \prod_{i=1}^{n-k+1}\biggl(\frac{x_i}{i!}\biggr)^{\ell_i}.
\end{equation}
This kind of polynomials are important in combinatorics, number theory, analysis, and other areas in mathematical sciences. In recent years, some new conclusions and applications of specific values for partial Bell polynomials $\bell_{n,k}$ have been reviewed and surveyed in~\cite{CDM-68111.tex, Bell-value-elem-funct.tex}. Our main result in Section~\ref{sec-Bell-special-values} simplifies and unifies those results in~~\cite[Sections~1 and~3]{AIMS-Math20210491.tex}.
\par
In Section~\ref{sec-series-representation-logsine}, applying Maclaurin's series expansion of the power $\bigl(\frac{\arcsin t}{t}\bigr)^m$ established in Section~\ref{sec-arcsin-series-expan}, we will derive a series representation of the generalized logsine function
\begin{equation}\label{generalized-logsine-dfn}
\ls_j^{(k)}(\theta)=-\int_{0}^{\theta}x^k\biggl(\ln\biggl|2\sin\frac{x}{2}\biggr|\biggr)^{j-k-1}\td x,
\end{equation}
where $j,k$ are integers with $j\ge k+1\in\mathbb{N}$ and $\theta$ is an arbitrary real number. The generalized logsine function $\ls_j^{(k)}(\theta)$ was originally introduced in~\cite[pp.~191--192]{Lewin-B-1981}. This series representation of the generalized logsine function $\ls_j^{(k)}(\theta)$ simplifies and unifies corresponding ones in~\cite[Section~4]{AIMS-Math20210491.tex} and~\cite{Davydychev-Kalmykov-2001, Kalmykov-Sheplyakov-lsjk-2005}.
\par
In Section~\ref{sec-hyperbolic-sine-id}, by similar methods used in Section~\ref{sec-arcsin-series-expan}, we will discover Maclaurin's series expansions of the positive integer powers $\bigl(\frac{\arcsinh t}{t}\bigr)^m$ for $m\in\mathbb{N}$, the composite $\Gamma(m,\arcsinh t)$ for $m\ge2$, and the exponential function $\te^{\arcsinh t}$, while we will also present three series identities involving $(\arcsinh t)^\ell$ for $\ell\ge2$ and the first kind Stirling numbers $s(n,k)$, where the incomplete gamma function $\Gamma(z,x)$ is defined~\cite{Jameson-MZ-2016, ingamma, qi-senlin-mia, Qi-Mei-99-gamma} by
\begin{equation}\label{incomplete-gamma-dfn}
\Gamma(z,x)=\int_{x}^{\infty}\te^{-t}t^{z-1}\td t
\end{equation}
for $\Re(z)>0$ and $x\in\mathbb{N}_0$.
\par
In Section~\ref{sec-2guesses}, basing on several known Maclaurin's series expansions for positive integer powers of the inverse tangent function $\arctan t$ and the inverse hyperbolic tangent function $\arctanh t$, we will guess and verify two explicit and general expressions of Maclaurin's series expansions of positive integer powers $(\arctan t)^n$ and $(\arctanh t)^n$ for $n\in\mathbb{N}$.
\par
In Section~\ref{sec-power-remarks}, we state useful remarks on our main results and related stuffs, including infinite series representations of positive integer powers of the circular constant $\pi$.

\section{Maclaurin's series expansion for positive integer powers of inverse sine function}\label{sec-arcsin-series-expan}

In~\cite[Remarks~5.2 to~5.5]{AIMS-Math20210491.tex}, there is a review and a survey of special cases of Maclaurin's series expansions of $(\arcsin t)^\ell$ for $\ell\in\mathbb{N}$. In~\cite[Section~2]{AIMS-Math20210491.tex}, general expressions for Maclaurin's series expansions of $(\arcsin t)^{2\ell-1}$ and $(\arcsin t)^{2\ell}$ for $\ell\in\mathbb{N}$ were established respectively.
\par
In this section, we discover a nice, simpler, and general expression of Maclaurin's series expansions of the functions $(\arcsin t)^{m}$ and $\frac{(\arcsin t)^{m}}{\sqrt{1-t^2}\,}$ for $m\in\mathbb{N}$ and derive two combinatorial identities for finite sums involving the first kind Stirling numbers $s(n,k)$.

\begin{thm}\label{arcsin-series-expansion-unify-thm}
For $m\in\mathbb{N}$ and $|t|<1$, the function $\bigl(\frac{\arcsin t}{t}\bigr)^{m}$, whose value at $t=0$ is defined to be $1$, has Maclaurin's series expansion
\begin{equation}\label{arcsin-series-expansion-unify}
\biggl(\frac{\arcsin t}{t}\biggr)^{m}
=1+\sum_{k=1}^{\infty} (-1)^k\frac{Q(m,2k;2)}{\binom{m+2k}{m}}\frac{(2t)^{2k}}{(2k)!},
\end{equation}
where
\begin{equation}\label{Q(m-k)-sum-dfn}
Q(m,k;\alpha)=\sum_{\ell=0}^{k} \binom{m+\ell-1}{m-1} s(m+k-1,m+\ell-1)\biggl(\frac{m+k-\alpha}{2}\biggr)^{\ell}
\end{equation}
for $m,k\in\mathbb{N}$ and $\alpha\in\mathbb{R}$ such that $m+k\ne\alpha$ and $s(m+k-1,m+\ell-1)$ is generalized by~\eqref{gen-funct-3}.
\end{thm}

\begin{proof}[First proof]
In~\cite[p.~3, (2.7)]{Borwein-Chamberland-IJMMS-2007} and~\cite[pp.~210--211, (10.49.33) and~(10.49.34)]{Hansen-B-1975}, there are the formulas
\begin{equation}\label{arcsin-pochhammer}
\sum_{k=0}^{\infty}\frac{(\ti a)_{k/2}}{(\ti a+1)_{-k/2}}\frac{(-\ti x)^k}{k!}
=\exp\biggl(2a\arcsin\frac{x}{2}\biggr)
\end{equation}
and
\begin{equation}\label{JO(833)}
\sum_{k=0}^{\infty}\frac{\bigl(\ti a+\frac{1}{2}\bigr)_{k/2}}{\bigl(\ti a+\frac{1}{2}\bigr)_{-k/2}} \frac{(-\ti x)^k}{k!}
=\frac{2}{\sqrt{4-x^2}\,}\exp\biggl(2a\arcsin\frac{x}{2}\biggr),
\end{equation}
where $\ti=\sqrt{-1}$ is the imaginary unit, the extended Pochhammer symbol $(z)_\alpha$ for $z,\alpha\in\mathbb{C}$ such that $z+\alpha\ne0,-1,-2,\dotsc$ is defined by
\begin{equation}\label{extended-Pochhammer-dfn}
(z)_{\alpha}=\frac{\Gamma(z+\alpha)}{\Gamma(z)},
\end{equation}
and the Euler gamma function $\Gamma(z)$ is defined~\cite[Chapter~3]{Temme-96-book} by
\begin{equation*}
\Gamma(z)=\lim_{n\to\infty}\frac{n!n^z}{\prod_{k=0}^n(z+k)}, \quad z\in\mathbb{C}\setminus\{0,-1,-2,\dotsc\}.
\end{equation*}
In the formulas~\eqref{arcsin-pochhammer} and~\eqref{JO(833)}, replacing $x$ by $2x$ and employing the extended Pochhammer symbol in~\eqref{extended-Pochhammer-dfn} gives
\begin{equation}\label{arcsin-pochhammer-binomial}
\begin{split}
\te^{2a\arcsin x}
&=\sum_{k=0}^{\infty}(-2\ti)^k \frac{\Gamma\bigl(\ti a+\frac{k}{2}\bigr)}{\Gamma(\ti a)} \frac{\Gamma(\ti a+1)}{\Gamma\bigl(\ti a-\frac{k}{2}+1\bigr)}\frac{x^k}{k!}\\
&=1+\ti a\sum_{k=1}^{\infty}(-2\ti)^k \binom{\ti a+\frac{k}{2}-1}{k-1}\frac{x^k}{k}
\end{split}
\end{equation}
and
\begin{equation}\label{JO(833)-x2(2x)}
\begin{aligned}
\frac{\te^{2a\arcsin x}}{\sqrt{1-x^2}\,}
&=\sum_{k=0}^{\infty} (-2\ti)^k\frac{\Gamma\bigl(\ti a+\frac{1+k}{2}\bigr)}{\Gamma\bigl(\ti a+\frac{1-k}{2}\bigr)}\frac{x^k}{k!}\\
&=\sum_{k=0}^{\infty} (-2\ti)^k\binom{\ti a+\frac{k-1}{2}}{k}x^k,
\end{aligned}
\end{equation}
where the extended binomial coefficient $\binom{z}{w}$ is defined by
\begin{equation}\label{Gen-Coeff-Binom}
\binom{z}{w}=
\begin{dcases}
\frac{\Gamma(z+1)}{\Gamma(w+1)\Gamma(z-w+1)}, & z\not\in\mathbb{N}_-,\quad w,z-w\not\in\mathbb{N}_-\\
0, & z\not\in\mathbb{N}_-,\quad w\in\mathbb{N}_- \text{ or } z-w\in\mathbb{N}_-\\
\frac{\langle z\rangle_w}{w!},& z\in\mathbb{N}_-, \quad w\in\mathbb{N}_0\\
\frac{\langle z\rangle_{z-w}}{(z-w)!}, & z,w\in\mathbb{N}_-, \quad z-w\in\mathbb{N}_0\\
0, & z,w\in\mathbb{N}_-, \quad z-w\in\mathbb{N}_-\\
\infty, & z\in\mathbb{N}_-, \quad w\not\in\mathbb{Z}
\end{dcases}
\end{equation}
in terms of the gamma function $\Gamma(z)$ and the falling factorial
\begin{equation}\label{Fall-Factorial-Dfn-Eq}
\langle z\rangle_k=
\prod_{\ell=0}^{k-1}(z-\ell)=
\begin{cases}
z(z-1)\dotsm(z-k+1), & k\in\mathbb{N};\\
1,& k=0.
\end{cases}
\end{equation}
Integrating on both sides of~\eqref{JO(833)-x2(2x)} with respect to $x\in(0,t)\subset(-1,1)$ leads to
\begin{equation}\label{JO(833)-x2(2x)-INT}
\frac{\te^{2a\arcsin t}-1}{2a}
=t+at^2+\sum_{k=2}^{\infty} (-2\ti)^k\binom{\ti a+\frac{k-1}{2}}{k}\frac{t^{k+1}}{k+1}.
\end{equation}
\par
In~\cite[p.~165, (12.1)]{Quaintance-Gould-2016-B}, there is the formula
\begin{equation}\label{(12.1)-Quaintance}
k!\binom{z}{k}=\sum_{\ell=0}^{k}s(k,\ell)z^\ell, \quad z\in\mathbb{C}.
\end{equation}
Therefore, we acquire
\begin{align*}
\binom{\ti a+\frac{k}{2}-1}{k-1}&=\frac{1}{(k-1)!}\sum_{\ell=0}^{k-1}s(k-1,\ell)\biggl(\ti a+\frac{k}{2}-1\biggr)^\ell\\
&=\frac{1}{(k-1)!}\sum_{\ell=0}^{k-1}s(k-1,\ell)\sum_{m=0}^{\ell}\binom{\ell}{m}(\ti a)^m\biggl(\frac{k-2}{2}\biggr)^{\ell-m}\\
&=\frac{1}{(k-1)!}\sum_{m=0}^{k-1}\Biggl[\sum_{\ell=m}^{k-1} \binom{\ell}{m} s(k-1,\ell) \biggl(\frac{k-2}{2}\biggr)^{\ell-m}\Biggr](\ti a)^m
\end{align*}
for $k\ge3$. Substituting this into the right hand side of~\eqref{arcsin-pochhammer-binomial} yields
\begin{align*}
&\quad 1+\ti a\sum_{k=1}^{\infty}(-2\ti)^k \binom{\ti a+\frac{k}{2}-1}{k-1}\frac{x^k}{k}\\
&=1+2ax+2a^2x^2+\sum_{k=3}^{\infty}(-2\ti)^k \sum_{m=0}^{k-1}\Biggl[\sum_{\ell=m}^{k-1} \binom{\ell}{m} s(k-1,\ell) \biggl(\frac{k-2}{2}\biggr)^{\ell-m}\Biggr](\ti a)^{m+1}\frac{x^k}{k!}\\
&=1+\Biggl[2x+\sum_{k=3}^{\infty}(-2)^k\ti^{k+1}Q(1,k-1;2)\frac{x^k}{k!}\Biggr]a\\
&\quad+\Biggl(2x^2-\sum_{k=3}^{\infty}(-2\ti)^k \Biggl[\sum_{\ell=1}^{k-1}\ell s(k-1,\ell) \biggl(\frac{k-2}{2}\biggr)^{\ell-1}\Biggr]\frac{x^k}{k!}\Biggr)a^{2}\\
&\quad+\sum_{m=3}^{\infty}\ti^m \Biggl(\sum_{k=m}^{\infty}(-2\ti)^k\Biggl[\sum_{\ell=m}^{k} \binom{\ell-1}{m-1} s(k-1,\ell-1) \biggl(\frac{k-2}{2}\biggr)^{\ell-m}\Biggr]\frac{x^k}{k!}\Biggr)a^{m}.
\end{align*}
The left hand side of~\eqref{arcsin-pochhammer-binomial} can be expanded into
\begin{equation*}
\te^{2a\arcsin x}=\sum_{m=0}^{\infty}\frac{(2a\arcsin x)^m}{m!}=\sum_{m=0}^{\infty}\frac{(2\arcsin x)^m}{m!}a^m.
\end{equation*}
Comparing coefficients of $a^m$ for $m\in\mathbb{N}$ results in
\begin{align}
\begin{split}\label{arcsin-FIRST-im}
2\arcsin x&=2x+\sum_{k=3}^{\infty}(-2)^k\ti^{k+1}Q(1,k-1;2)\frac{x^k}{k!}\\
&=2x\Biggl[1+\sum_{k=1}^{\infty}\frac{(-1)^{k}}{2k+1}Q(1,2k;2) \frac{(2x)^{2k}}{(2k)!}\Biggr],
\end{split}\\
\begin{split}\label{arcsin-square-FIRST-im}
\frac{(2\arcsin x)^2}{2!}&=2x^2-\sum_{k=3}^{\infty}(-2\ti)^k \Biggl[\sum_{\ell=1}^{k-1}\ell s(k-1,\ell) \biggl(\frac{k-2}{2}\biggr)^{\ell-1}\Biggr]\frac{x^k}{k!}\\
&=2x^2\Biggl[1+\sum_{k=1}^{\infty}\frac{2(-1)^{k}}{(2k+2)(2k+1)} Q(2,2k;2) \frac{(2x)^{2k}}{(2k)!}\Biggr],
\end{split}
\end{align}
and
\begin{equation}
\begin{aligned}\label{arcsin-Power-FIRST-im}
\frac{(2\arcsin x)^m}{m!}
&=\ti^m\sum_{k=m}^{\infty}(-2\ti)^k\Biggl[\sum_{\ell=m}^{k} \binom{\ell-1}{m-1} s(k-1,\ell-1) \biggl(\frac{k-2}{2}\biggr)^{\ell-m}\Biggr]\frac{x^k}{k!}\\
&=\frac{(2x)^m}{m!}\sum_{k=0}^{\infty}\frac{(-1)^{k}}{\binom{m+2k}{m}}Q(m,2k;2) \frac{(2x)^{2k}}{(2k)!}
\end{aligned}
\end{equation}
for $m\ge3$. Maclaurin's series expansion~\eqref{arcsin-series-expansion-unify} in Theorem~\ref{arcsin-series-expansion-unify-thm} is proved. The first proof of Theorem~\ref{arcsin-series-expansion-unify-thm} is complete.
\end{proof}

\begin{proof}[Second proof]
By virtue of~\eqref{(12.1)-Quaintance} again, we obtain
\begin{align*}
\binom{\ti a+\frac{k-1}{2}}{k}&=\frac{1}{k!}\sum_{\ell=0}^{k}s(k,\ell)\biggl(\ti a+\frac{k-1}{2}\biggr)^\ell\\
&=\frac{1}{k!}\sum_{\ell=0}^{k}s(k,\ell)\sum_{m=0}^{\ell} \binom{\ell}{m}(\ti a)^m \biggl(\frac{k-1}{2}\biggr)^{\ell-m}\\
&=\frac{1}{k!}\sum_{m=0}^{k}\Biggl[\sum_{\ell=m}^{k}\binom{\ell}{m} s(k,\ell) \biggl(\frac{k-1}{2}\biggr)^{\ell-m}\Biggr](\ti a)^m
\end{align*}
for $k\ge2$. Substituting this into the right hand side of~\eqref{JO(833)-x2(2x)-INT} results in
\begin{align*}%\label{right-series-expansion}
&\quad t+at^2+\sum_{k=2}^{\infty} (-2\ti)^k\binom{\ti a+\frac{k-1}{2}}{k}\frac{t^{k+1}}{k+1}\\
&=t+at^2+\sum_{k=2}^{\infty} (-2)^k\Biggl(\sum_{m=0}^{k}\ti^{k+m}\Biggl[\sum_{\ell=m}^{k} \binom{\ell}{m} s(k,\ell) \biggl(\frac{k-1}{2}\biggr)^{\ell-m}\Biggr] a^m\Biggr) \frac{t^{k+1}}{(k+1)!}\\
&=t+at^2+\sum_{k=2}^{\infty} (-2)^k \ti^{k}Q(1,k;2)\frac{t^{k+1}}{(k+1)!}\\
&\quad+\sum_{k=2}^{\infty} (-2)^k\Biggl(\Biggl[\ti^{k+1}\sum_{\ell=1}^{k}\ell s(k,\ell) \biggl(\frac{k-1}{2}\biggr)^{\ell-1}\Biggr]a\Biggr) \frac{t^{k+1}}{(k+1)!}\\
&\quad+\sum_{k=2}^{\infty} (-2)^k\Biggl(\sum_{m=2}^{k}\Biggl[\ti^{k+m}\sum_{\ell=m}^{k} \binom{\ell}{m} s(k,\ell) \biggl(\frac{k-1}{2}\biggr)^{\ell-m}\Biggr]a^m\Biggr) \frac{t^{k+1}}{(k+1)!}\\
&=t+\sum_{k=2}^{\infty} (-2)^k \ti^{k}Q(1,k;2) \frac{t^{k+1}}{(k+1)!} \\
&\quad+\Biggl(t^2+\sum_{k=2}^{\infty}\ti^{k+1}(-2)^k\Biggl[\sum_{\ell=1}^{k}\ell s(k,\ell) \biggl(\frac{k-1}{2}\biggr)^{\ell-1}\Biggr] \frac{t^{k+1}}{(k+1)!}\Biggr)a\\
&\quad+\sum_{m=2}^{\infty}\Biggl(\sum_{k=m}^{\infty} (-2)^k\Biggl[\ti^{k+m}\sum_{\ell=m}^{k} \binom{\ell}{m} s(k,\ell) \biggl(\frac{k-1}{2}\biggr)^{\ell-m}\Biggr]\frac{t^{k+1}}{(k+1)!}\Biggr)a^m.
\end{align*}
The series expansion of the left hand side in~\eqref{JO(833)-x2(2x)-INT} is
\begin{equation}\label{left-series-expansion}
\frac{\te^{2a\arcsin t}-1}{2a}
=\sum_{m=1}^{\infty}\frac{(2a)^{m-1}(\arcsin t)^m}{m!}
=\sum_{m=0}^{\infty}\frac{2^{m}(\arcsin t)^{m+1}}{(m+1)!}a^{m}.
\end{equation}
Equating coefficients of $a^m$ for $m\in\mathbb{N}_0$ in the series~\eqref{left-series-expansion} and its previous one produces
\begin{align}
\begin{split}\label{arcsin-im}
\arcsin t&=t+\sum_{k=2}^{\infty} (-2)^k \ti^{k}Q(1,k;2) \frac{t^{k+1}}{(k+1)!}\\
&=t\Biggl[1+\sum_{k=1}^{\infty} (-2)^{2k}\ti^{2k}Q(1,2k;2)\frac{t^{2k}}{(2k+1)!}\Biggr]\\
&=t\Biggl[1+\sum_{k=1}^{\infty}\frac{(-1)^{k}}{2k+1}Q(1,2k;2) \frac{(2t)^{2k}}{(2k)!}\Biggr],
\end{split}\\
\begin{split}\label{arcsin-square-im}
(\arcsin t)^2&=t^2+\sum_{k=2}^{\infty}\ti^{k+1}(-2)^k\Biggl[\sum_{\ell=1}^{k}\ell s(k,\ell) \biggl(\frac{k-1}{2}\biggr)^{\ell-1}\Biggr] \frac{t^{k+1}}{(k+1)!}\\
&=t^2\Biggl(1+\sum_{k=1}^{\infty}\ti^{2k+2}(-2)^{2k+1}\Biggl[\sum_{\ell=1}^{2k+1}\ell s(2k+1,\ell) \biggl(\frac{2k}{2}\biggr)^{\ell-1}\Biggr] \frac{t^{2k}}{(2k+2)!}\Biggr)\\
&=t^2\Biggl[1+\sum_{k=1}^{\infty}\frac{(-1)^{k}}{(k+1)(2k+1)}Q(2,2k;2) \frac{(2t)^{2k}}{(2k)!}\Biggr],
\end{split}
\end{align}
and
\begin{equation}\label{arcsin-Power-m-im}
\begin{aligned}
\frac{2^{m}(\arcsin t)^{m+1}}{(m+1)!}
&=\frac{2^{m}t^{m+1}}{(m+1)!}\Biggl((-1)^m\ti^{m}\sum_{k=m}^{\infty} (m+1)!2^{k-m}\\
&\quad\times\Biggl[\ti^{k}\sum_{\ell=m}^{k} \binom{\ell}{m} s(k,\ell) \biggl(\frac{k-1}{2}\biggr)^{\ell-m}\Biggr]\frac{t^{k-m}}{(k+1)!}\Biggr)\\
&=\frac{2^{m}t^{m+1}}{(m+1)!}\Biggl(1+(-1)^m\ti^{m}\sum_{k=m+1}^{\infty} (m+1)!2^{k-m}\\
&\quad\times\Biggl[\ti^{k}\sum_{\ell=m}^{k} \binom{\ell}{m} s(k,\ell) \biggl(\frac{k-1}{2}\biggr)^{\ell-m}\Biggr]\frac{t^{k-m}}{(k+1)!}\Biggr)\\
&=\frac{2^{m}t^{m+1}}{(m+1)!}\Biggl[1+\sum_{k=1}^{\infty} (m+1)! \ti^{k}Q(m+1,k;2)\frac{(2t)^{k}}{(k+m+1)!}\Biggr]\\
&=\frac{2^{m}t^{m+1}}{(m+1)!}\Biggl[1+\sum_{k=1}^{\infty} (-1)^k\frac{(m+1)!(2k)!}{(2k+m+1)!} Q(m+1,2k;2)\frac{(2t)^{2k}}{(2k)!}\Biggr]
\end{aligned}
\end{equation}
for $m\ge2$.
Maclaurin's series expansion~\eqref{arcsin-series-expansion-unify} in Theorem~\ref{arcsin-series-expansion-unify-thm} is proved once again. The second proof of Theorem~\ref{arcsin-series-expansion-unify-thm} is complete.
\end{proof}

\begin{cor}\label{arcsin-diiff-series-cor}
For $m\in\mathbb{N}_0$ and $|t|<1$, we have Maclaurin's series expansion
\begin{equation}\label{arcsin-diiff-series}
\frac{(\arcsin t)^{m}}{\sqrt{1-t^2}\,}
=t^m\Biggl[1+\sum_{k=1}^{\infty} (-1)^k\frac{Q(m+1,2k;2)}{\binom{m+2k}{m}}\frac{(2t)^{2k}}{(2k)!}\Biggr],
\end{equation}
where $Q(m,k;2)$ is defined by~\eqref{Q(m-k)-sum-dfn}.
\end{cor}

\begin{proof}
Multiplying by $t^m$ and differentiating with respect to $t$ on both sides of~\eqref{arcsin-series-expansion-unify} in sequence yield
\begin{equation*}
\frac{m(\arcsin t)^{m-1}}{\sqrt{1-t^2}\,}
=mt^{m-1}+\sum_{k=1}^{\infty} \frac{(-1)^k}{\binom{2k+m}{m}}
Q(m,2k;2)\frac{2^{2k}(2k+m)t^{2k+m-1}}{(2k)!}.
\end{equation*}
Replacing $m-1$ by $m$ and simplifying lead to~\eqref{arcsin-diiff-series}.
Corollary~\ref{arcsin-diiff-series-cor} is thus proved.
\end{proof}

\begin{cor}\label{arcsin-deriv-comb-id-cor}
For $k,m\in\mathbb{N}$, we have the combinatorial identities
\begin{equation}\label{1st-stirling-2k+1}
Q(1,2k+1;2)=0
\end{equation}
and
\begin{equation}\label{3rd-comb-id}
Q(m+1,2k-1;2)=0,
\end{equation}
where $Q(m,k;2)$ is defined by~\eqref{Q(m-k)-sum-dfn}.
\end{cor}

\begin{proof}
This follows from the disappearance of imaginary parts in~\eqref{arcsin-FIRST-im}, \eqref{arcsin-square-FIRST-im}, \eqref{arcsin-Power-FIRST-im}, \eqref{arcsin-im}, \eqref{arcsin-square-im}, and~\eqref{arcsin-Power-m-im} and from reformulation.
Corollary~\ref{arcsin-deriv-comb-id-cor} is thus proved.
\end{proof}

\section{Application to specific values of partial Bell polynomials}\label{sec-Bell-special-values}

The specific values of partial Bell polynomials $\bell_{2n,k}$ in~\eqref{special-values-Oertel} have been represented in~\cite[Sections~1 and~3]{AIMS-Math20210491.tex} by two explicit formulas for two cases $\bell_{2n,2k-1}$ and $\bell_{2n,2k}$ respectively. In this section, applying Maclaurin's series expansion~\eqref{arcsin-series-expansion-unify} in Theorem~\ref{arcsin-series-expansion-unify-thm}, we give a simpler, nicer, unified, and closed-form formula of $\bell_{2n,k}$ in~\eqref{special-values-Oertel} in terms of the quantity $Q(m,k;2)$ defined in~\eqref{Q(m-k)-sum-dfn}.

\begin{thm}\label{Bell-Oertel-closed-thm}
For $k,n\in\mathbb{N}$ such that $2n\ge k\in\mathbb{N}$, we have
\begin{multline}\label{Bell-Oertel-closed-Eq}
\bell_{2n,k}\biggl(0,\frac{1}{3},0,\frac{9}{5},0,\frac{225}{7},\dotsc, \frac{1+(-1)^{k+1}}{2}\frac{[(2n-k)!!]^2}{2n-k+2}\biggr)\\*
=(-1)^{n+k}\frac{(4n)!!}{(2n+k)!}\sum_{q=1}^{k}(-1)^{q}\binom{2n+k}{k-q} Q(q,2n;2),
\end{multline}
where $Q(q,2n;2)$ is given by~\eqref{Q(m-k)-sum-dfn}.
\end{thm}

\begin{proof}
It is well known that the power series expansion
\begin{equation*}
\arcsin t=\sum_{\ell=0}^{\infty}[(2\ell-1)!!]^2\frac{t^{2\ell+1}}{(2\ell+1)!},\quad |t|<1
\end{equation*}
is valid, where $(-1)!!=1$.
This implies that
\begin{gather*}
\bell_{2n,k}\biggl(0,\frac{1}{3},0,\frac{9}{5},0,\frac{225}{7},\dotsc, \frac{1+(-1)^{k+1}}{2}\frac{[(2n-k)!!]^2}{2n-k+2}\biggr)\\
=\bell_{2n,k}\biggl(\frac{(\arcsin t)''|_{t=0}}{2}, \frac{(\arcsin t)'''|_{t=0}}{3}, \frac{(\arcsin t)^{(4)}|_{t=0}}{4},\dotsc, \frac{(\arcsin t)^{(2n-k+2)}|_{t=0}}{2n-k+2}\biggr).
\end{gather*}
Employing the formula
\begin{equation*}
\bell_{n,k}\biggl(\frac{x_2}{2},\frac{x_3}{3},\dotsc,\frac{x_{n-k+2}}{n-k+2}\biggr)
=\frac{n!}{(n+k)!}\bell_{n+k,k}(0,x_2,x_3,\dotsc,x_{n+1})
\end{equation*}
in~\cite[p.~136]{Comtet-Combinatorics-74}, we acquire
\begin{gather*}
\bell_{2n,k}\biggl(0,\frac{1}{3},0,\frac{9}{5},0,\frac{225}{7},\dotsc, \frac{1+(-1)^{k+1}}{2}\frac{[(2n-k)!!]^2}{2n-k+2}\biggr)\\
=\frac{(2n)!}{(2n+k)!}\bell_{2n+k,k}\bigl(0,(\arcsin t)''|_{t=0},(\arcsin t)'''|_{t=0},\dotsc,(\arcsin t)^{(2n+1)}|_{t=0}\bigr).
\end{gather*}
\par
Making use of the formula
\begin{equation*}%\label{113-final-formula}
\frac1{k!}\Biggl(\sum_{m=1}^\infty x_m\frac{t^m}{m!}\Biggr)^k =\sum_{n=k}^\infty \bell_{n,k}(x_1,x_2,\dotsc,x_{n-k+1})\frac{t^n}{n!}
\end{equation*}
for $k\in\mathbb{N}_0$ in~\cite[p.~133]{Comtet-Combinatorics-74} yields
\begin{align*}
\sum_{n=0}^\infty \bell_{n+k,k}(x_1,x_2,\dotsc,x_{n+1})\frac{k!n!}{(n+k)!}\frac{t^{n+k}}{n!}
&=\Biggl(\sum_{m=1}^\infty x_m\frac{t^m}{m!}\Biggr)^k,\\
\sum_{n=0}^\infty \frac{\bell_{n+k,k}(x_1,x_2,\dotsc,x_{n+1})}{\binom{n+k}{k}}\frac{t^{n+k}}{n!}
&=\Biggl(\sum_{m=1}^\infty x_m\frac{t^m}{m!}\Biggr)^k,\\
\bell_{n+k,k}(x_1,x_2,\dotsc,x_{n+1})
&=\binom{n+k}{k}\lim_{t\to0}\frac{\td^n}{\td t^n}\Biggl[\sum_{m=0}^\infty x_{m+1}\frac{t^{m}}{(m+1)!}\Biggr]^k,\\
\bell_{2n+k,k}(x_1,x_2,\dotsc,x_{2n+1})
&=\binom{2n+k}{k}\lim_{t\to0}\frac{\td^{2n}}{\td t^{2n}}\Biggl[\sum_{m=0}^\infty x_{m+1}\frac{t^{m}}{(m+1)!}\Biggr]^k.
\end{align*}
Setting $x_1=0$ and $x_m=(\arcsin t)^{(m)}|_{t=0}$ for $m\ge2$ gives
\begin{align*}
&\quad\bell_{2n+k,k}\bigl(0,(\arcsin t)''|_{t=0},(\arcsin t)'''|_{t=0},\dotsc,(\arcsin t)^{(2n+1)}|_{t=0}\bigr)\\
&=\binom{2n+k}{k}\frac{\td^{2n}}{\td t^{2n}}\Biggl[\frac{1}{t} \sum_{m=2}^\infty (\arcsin t)^{(m)}|_{t=0} \frac{t^m}{m!}\Biggr]^k\\
&=\binom{2n+k}{k}\frac{\td^{2n}}{\td t^{2n}}\biggl(\frac{\arcsin t-t}{t}\biggr)^k\\
&=\binom{2n+k}{k}\frac{\td^{2n}}{\td t^{2n}}\sum_{q=0}^{k}(-1)^{k-q}\binom{k}{q}\biggl(\frac{\arcsin t}{t}\biggr)^q\\
&=\binom{2n+k}{k}\sum_{q=1}^{k}(-1)^{k-q}\binom{k}{q}\frac{\td^{2n}}{\td t^{2n}}\biggl(\frac{\arcsin t}{t}\biggr)^q.
\end{align*}
By virtue of the series expansion~\eqref{arcsin-series-expansion-unify} in Theorem~\ref{arcsin-series-expansion-unify-thm}, we obtain
\begin{equation*}
\lim_{t\to0}\frac{\td^{2n}}{\td t^{2n}}\biggl(\frac{\arcsin t}{t}\biggr)^q
=(-1)^n\frac{2^{2n}}{\binom{2n+q}{q}} Q(q,2n;2)
\end{equation*}
for $n\ge q\in\mathbb{N}$. In conclusion, we arrive at
\begin{align*}
&\quad\bell_{2n,k}\biggl(0,\frac{1}{3},0,\frac{9}{5},0,\frac{225}{7},\dotsc, \frac{1+(-1)^{k+1}}{2}\frac{[(2n-k)!!]^2}{2n-k+2}\biggr)\\
&=\frac{(-1)^n(2n)!}{(2n+k)!}\binom{2n+k}{k}\sum_{q=1}^{k}\binom{k}{q} \frac{(-1)^{k-q}2^{2n}}{\binom{2n+q}{q}} Q(q,2n;2)\\
&=(-1)^{n+k}\frac{(4n)!!}{(2n+k)!}\sum_{q=1}^{k}(-1)^{q}\binom{2n+k}{k-q} Q(q,2n;2).
\end{align*}
The proof of Theorem~\ref{Bell-Oertel-closed-thm} is complete.
\end{proof}

\section{Application to series representation of generalized logsine function}\label{sec-series-representation-logsine}

In~\cite[Section~4 and Remark~5.7]{AIMS-Math20210491.tex}, two series representations for generalized logsine function $\ls_j^{(k)}(\theta)$ defined by~\eqref{generalized-logsine-dfn} were established by two cases $\ls_j^{(2\ell-1)}(\theta)$ and $\ls_j^{(2\ell)}(\theta)$ for $\ell\in\mathbb{N}$ respectively. In this section, applying Maclaurin's series expansion~\eqref{arcsin-series-expansion-unify} in Theorem~\ref{arcsin-series-expansion-unify-thm}, we derive a simpler and unified series representation of $\ls_j^{(k)}(\theta)$ for $k\in\mathbb{N}$ in terms of the quantity $Q(m,k;2)$ defined in~\eqref{Q(m-k)-sum-dfn}.

\begin{thm}\label{logsine-series-expnsion-thm}
In the region $0<\theta\le\pi$ and for $j,k\in\mathbb{N}$, generalized logsine function $\ls_j^{(k)}(\theta)$ has the series representation
\begin{equation}\label{logsine-series-expnsion-represent}
\begin{aligned}
\ls_j^{(k)}(\theta)
&=(\ln2)^{j}\biggl(\frac{2\sin\frac{\theta}{2}}{\ln2}\biggr)^{k+1} \Biggl[k!\sum_{q=1}^{\infty} \frac{(-1)^{q+1}}{(k+2q)!}\biggl(2\sin\frac{\theta}{2}\biggr)^{2q}Q(k+1,2q;2)\\
&\quad\times\sum_{\ell=0}^{j-k-1}\binom{j-k-1}{\ell}\biggl(\frac{\ln\sin\frac{\theta}{2}}{\ln2}\biggr)^{\ell} \sum_{p=0}^{\ell}\frac{(-1)^p\langle\ell\rangle_{p}} {(k+2q+1)^{p+1}\bigl(\ln\sin\frac{\theta}{2}\bigr)^{p}}\\
&\quad-\sum_{\ell=0}^{j-k-1}\binom{j-k-1}{\ell} \biggl(\frac{\ln\sin\frac{\theta}{2}}{\ln2}\biggr)^{\ell} \sum_{p=0}^{\ell} \frac{(-1)^p\langle\ell\rangle_{p}}{(k+1)^{p+1}\bigl(\ln\sin\frac{\theta}{2}\bigr)^{p}}\Biggr],
\end{aligned}
\end{equation}
where the falling factorial $\langle z\rangle_k$ is defined by~\eqref{Fall-Factorial-Dfn-Eq} and $Q(k+1,2q;2)$ is defined by~\eqref{Q(m-k)-sum-dfn}.
\end{thm}

\begin{proof}
In~\cite[p.~308]{Davydychev-Kalmykov-2001}, it was derived that
\begin{equation}\label{logsine-Diff-Part-Two}
\ls_j^{(k)}(\theta)=-2^{k+1}\int_{0}^{\sin(\theta/2)}\frac{(\arcsin x)^k}{\sqrt{1-x^2}\,}\ln^{j-k-1}(2x)\td x
\end{equation}
for $0<\theta\le\pi$ and $j\ge k+1\in\mathbb{N}$. Applying the series expansion~\eqref{arcsin-diiff-series} in Corollary~\ref{arcsin-diiff-series-cor} to~\eqref{logsine-Diff-Part-Two} gives
\begin{align*}
\ls_j^{(k)}(\theta)&=\sum_{q=1}^{\infty} \frac{(-1)^{q+1}}{{(2q)!}} \frac{2^{2q+k+1}}{\binom{2q+k}{k}}Q(k+1,2q;2) \int_{0}^{\sin(\theta/2)}x^{2q+k}\ln^{j-k-1}(2x)\td x\\
&\quad-2^{k+1}\int_{0}^{\sin(\theta/2)}x^k\ln^{j-k-1}(2x)\td x.
\end{align*}
Making use of the formula
\begin{equation*}%\label{x-power-log-int-eq}
\int x^n\ln^mx\td x=x^{n+1}\sum_{p=0}^{m}(-1)^p\langle m\rangle_{p}\frac{\ln^{m-p}x}{(n+1)^{p+1}}, \quad m,n\in\mathbb{N}_0
\end{equation*}
in~\cite[p.~238, 2.722]{Gradshteyn-Ryzhik-Table-8th} results in
\begin{gather*}
\int_{0}^{\sin(\theta/2)}x^{2q+k}\ln^{j-k-1}(2x)\td x
=\int_{0}^{\sin(\theta/2)}x^{2q+k}(\ln2+\ln x)^{j-k-1}\td x\\
=\int_{0}^{\sin(\theta/2)}x^{2q+k}\sum_{\ell=0}^{j-k-1}\binom{j-k-1}{\ell}(\ln2)^{j-k-\ell-1}(\ln x)^\ell\td x\\
=\sum_{\ell=0}^{j-k-1}\binom{j-k-1}{\ell}(\ln2)^{j-k-\ell-1}\int_{0}^{\sin(\theta/2)}x^{2q+k}(\ln x)^\ell\td x\\
=\sum_{\ell=0}^{j-k-1}\binom{j-k-1}{\ell}(\ln2)^{j-k-\ell-1} \biggl(\sin\frac{\theta}{2}\biggr)^{2q+k+1}\sum_{p=0}^{\ell}(-1)^p\langle \ell\rangle_{p}\frac{\bigl(\ln\sin\frac{\theta}{2}\bigr)^{\ell-p}}{(2q+k+1)^{p+1}}
\end{gather*}
and
\begin{equation*}
\int_{0}^{\sin(\theta/2)}x^k\ln^{j-k-1}(2x)\td x
=\sum_{\ell=0}^{j-k-1}\binom{j-k-1}{\ell}(\ln2)^{j-\ell} \biggl(\frac{\sin\frac{\theta}{2}}{\ln2}\biggr)^{k+1}\sum_{p=0}^{\ell}(-1)^p\langle \ell\rangle_{p}\frac{\bigl(\ln\sin\frac{\theta}{2}\bigr)^{\ell-p}}{(k+1)^{p+1}}.
\end{equation*}
Consequently, it follows that
\begin{align*}
\ls_j^{(k)}(\theta)&=\sum_{q=1}^{\infty} \frac{(-1)^{q+1}}{{(2q)!}} \frac{2^{2q+k+1}}{\binom{2q+k}{k}}Q(k+1,2q;2)\\
&\quad\times\sum_{\ell=0}^{j-k-1}\binom{j-k-1}{\ell}(\ln2)^{j-\ell} \biggl(\frac{\sin\frac{\theta}{2}}{\ln2}\biggr)^{2q+k+1}\sum_{p=0}^{\ell}(-1)^p\langle \ell\rangle_{p}\frac{\bigl(\ln\sin\frac{\theta}{2}\bigr)^{\ell-p}}{(2q+k+1)^{p+1}}\\
&\quad-2^{k+1}\sum_{\ell=0}^{j-k-1}\binom{j-k-1}{\ell}(\ln2)^{j-\ell} \biggl(\frac{\sin\frac{\theta}{2}}{\ln2}\biggr)^{k+1}\sum_{p=0}^{\ell}(-1)^p\langle \ell\rangle_{p}\frac{\bigl(\ln\sin\frac{\theta}{2}\bigr)^{\ell-p}}{(k+1)^{p+1}}\\
&=(\ln2)^{j}k!\biggl(\frac{2\sin\frac{\theta}{2}}{\ln2}\biggr)^{k+1} \sum_{q=1}^{\infty} \frac{(-1)^{q+1}}{(2q+k)!}\biggl(2\sin\frac{\theta}{2}\biggr)^{2q}Q(k+1,2q;2)\\
&\quad\times\sum_{\ell=0}^{j-k-1}\binom{j-k-1}{\ell}\biggl(\frac{\ln\sin\frac{\theta}{2}}{\ln2}\biggr)^{\ell} \sum_{p=0}^{\ell}\frac{(-1)^p\langle\ell\rangle_{p}} {(2q+k+1)^{p+1}\bigl(\ln\sin\frac{\theta}{2}\bigr)^{p}}\\
&\quad-(\ln2)^{j}\biggl(\frac{2\sin\frac{\theta}{2}}{\ln2}\biggr)^{k+1} \sum_{\ell=0}^{j-k-1}\binom{j-k-1}{\ell} \biggl(\frac{\ln\sin\frac{\theta}{2}}{\ln2}\biggr)^{\ell} \sum_{p=0}^{\ell} \frac{(-1)^p\langle\ell\rangle_{p}}{(k+1)^{p+1}\bigl(\ln\sin\frac{\theta}{2}\bigr)^{p}}.
\end{align*}
The proof of Theorem~\ref{logsine-series-expnsion-thm} is complete.
\end{proof}

\section{Maclaurin's series expansions for positive integer powers of inverse hyperbolic sine function and for incomplete gamma function}\label{sec-hyperbolic-sine-id}

In this section, by similar methods and arguments used in Section~\ref{sec-arcsin-series-expan}, we discover Maclaurin's series expansions of the function $\bigl(\frac{\arcsinh t}{t}\bigr)^m$ for $m\in\mathbb{N}$, of the function $\Gamma(m,\arcsinh t)$ for $m\ge2$, and of the exponential function $\te^{\arcsinh t}$, while we present three series identities involving $(\arcsinh t)^\ell$ for $\ell\ge2$ and the quantity $Q(m,k;\alpha)$ for $\alpha=2,3$.

\begin{thm}\label{arcsinh-identity-thm}
For $m\in\mathbb{N}$ and $|t|<\infty$, the function $\bigl(\frac{\arcsinh t}{t}\bigr)^{m}$, whose value at $t=0$ is defined to be $1$, has Maclaurin's series expansion
\begin{equation}\label{arcsinh-series-expansion}
\biggl(\frac{\arcsinh t}{t}\biggr)^m
=1+\sum_{k=1}^{\infty}\frac{Q(m,2k;2)}{\binom{m+2k}{m}} \frac{(2t)^{2k}}{(2k)!},
\end{equation}
where $Q(m,2k;2)$ is defined by~\eqref{Q(m-k)-sum-dfn}.
\par
The inverse hyperbolic sine function $\arcsinh t$ satisfies the series identities
\begin{align}\label{arcsinh-series-id-1}
\sum_{\ell=0}^{\infty}(-1)^{\ell}(\ell+1)\frac{(\arcsinh t)^{\ell+2}}{(\ell+2)!}
&=\frac{1}{2}t^2-\frac{1}{3}t^3+\frac{1}{4}\sum_{k=3}^{\infty} Q(2,k-1;3)\frac{(2t)^{k+1}}{(k+1)!},\\
\sum_{\ell=0}^{\infty}(-1)^{\ell}(\ell+1)(\ell+2)\frac{(\arcsinh t)^{\ell+3}}{(\ell+3)!} &=\frac{1}{3}t^3+\frac{1}{8}\sum_{k=3}^{\infty}Q(3,k-2;3)\frac{(2t)^{k+1}}{(k+1)!},
\label{arcsinh-series-id-2}
\end{align}
and, for $m\ge3$,
\begin{equation}\label{arcsinh-series-id-3}
\sum_{\ell=0}^{\infty}(-1)^{\ell}\binom{\ell+m}{m}\frac{(\arcsinh t)^{\ell+m+1}}{(\ell+m+1)!} =\frac{1}{2^{m+1}}\sum_{k=m}^{\infty}Q(m+1,k-m;3) \frac{(2t)^{k+1}}{(k+1)!},
\end{equation}
where $Q(m,k;3)$ is defined by~\eqref{Q(m-k)-sum-dfn}.
\par
The exponential function $\te^{\arcsinh t}$ of the inverse hyperbolic sine function $\arcsinh t$ has Maclaurin's series expansion
\begin{equation}\label{exp-arcsinh-sereis-exapnsion}
\te^{\arcsinh t}=1+t-t^2\sum_{k=0}^{\infty}\binom{\frac{2k-1}{2}}{2k+1}\frac{(2t)^{2k}}{k+1},
\end{equation}
where extended binomial coefficient $\binom{z}{w}$ is defined by~\eqref{Gen-Coeff-Binom}.
\end{thm}

\begin{proof}
In~\cite[pp.~210--211, (10.49.32) and~(10.49.35)]{Hansen-B-1975}, there are the formulas
\begin{equation}\label{(10.49.32)-first}
\sum_{k=0}^{\infty}\frac{(a)_{k/2}}{(a+1)_{-k/2}} \frac{x^k}{k!}
=\exp\Bigl(2a\arcsinh\frac{x}{2}\Bigr)
\end{equation}
and
\begin{equation}\label{(10.49.35)-first}
\sum_{k=0}^{\infty}\frac{(a)_{k/2}}{(a)_{-k/2}} \frac{x^k}{k!}
=\frac{2}{\sqrt{4+x^2}\,}\exp\Bigl[(2a-1)\arcsinh\frac{x}{2}\Bigr].
\end{equation}
In~\eqref{(10.49.32)-first} and~\eqref{(10.49.35)-first}, replacing $x$ by $2x$, utilizing extended Pochhammer symbol in~\eqref{extended-Pochhammer-dfn}, and employing extended binomial coefficient in~\eqref{Gen-Coeff-Binom} result in
\begin{equation}\label{(10.49.32)-rplace}
\te^{2a\arcsinh x}=a\sum_{k=0}^{\infty}\frac{\Gamma\bigl(a+\frac{k}{2}\bigr)}{\Gamma\bigl(a+1-\frac{k}{2}\bigr)}\frac{(2x)^k}{k!}
=1+a\sum_{k=1}^{\infty}\binom{a-1+\frac{k}{2}}{k-1}\frac{(2x)^k}{k}
\end{equation}
and
\begin{equation}\label{(10.49.35)-first-replace}
\frac{\exp[(2a-1)\arcsinh x]}{\sqrt{1+x^2}\,}
=\sum_{k=0}^{\infty}\frac{\Gamma\bigl(a+\frac{k}{2}\bigr)}{\Gamma\bigl(a-\frac{k}{2}\bigr)} \frac{(2x)^k}{k!}
=\sum_{k=0}^{\infty}\binom{a-1+\frac{k}{2}}{k}(2x)^k.
\end{equation}
Integrating on both sides of~\eqref{(10.49.35)-first-replace} with respect to $x\in(0,t)$ produces
\begin{equation}\label{(10.49.35)-first-replace-INT}
\frac{\exp[(2a-1)\arcsinh t]-1}{2a-1}
=\sum_{k=0}^{\infty}\binom{a-1+\frac{k}{2}}{k} \frac{2^kt^{k+1}}{k+1}.
\end{equation}
By virtue of the formula~\eqref{(12.1)-Quaintance}, we obtain
\begin{align*}
\binom{a-1+\frac{k}{2}}{k-1}&=\frac{1}{(k-1)!}\sum_{\ell=0}^{k-1}s(k-1,\ell)\biggl(a-1+\frac{k}{2}\biggr)^\ell\\
&=\frac{1}{(k-1)!}\sum_{\ell=0}^{k-1}s(k-1,\ell)\sum_{m=0}^{\ell}\binom{\ell}{m}\biggl(\frac{k-2}{2}\biggr)^{\ell-m}a^m\\
&=\frac{1}{(k-1)!}\sum_{m=0}^{k-1} \Biggl[\sum_{\ell=m}^{k-1}\binom{\ell}{m}s(k-1,\ell)\biggl(\frac{k-2}{2}\biggr)^{\ell-m}\Biggr]a^m\\
&=\frac{1}{(k-1)!}\sum_{m=0}^{k-1} Q(m+1,k-m-1;2)a^m
\end{align*}
and
\begin{align*}
\binom{a-1+\frac{k}{2}}{k}&=\frac{1}{k!}\sum_{\ell=0}^{k}s(k,\ell)\biggl(a-1+\frac{k}{2}\biggr)^\ell\\
&=\frac{1}{k!}\sum_{\ell=0}^{k}s(k,\ell)\sum_{m=0}^{\ell}\binom{\ell}{m}\biggl(\frac{k-2}{2}\biggr)^{\ell-m}a^m\\
&=\frac{1}{k!}\sum_{m=0}^{k} \Biggl[\sum_{\ell=m}^{k}\binom{\ell}{m}s(k,\ell)\biggl(\frac{k-2}{2}\biggr)^{\ell-m}\Biggr]a^m\\
&=\frac{1}{k!}\sum_{m=0}^{k} Q(m+1,k-m;3)a^m
\end{align*}
for $k\ge3$. Substituting these two finite sums into the right hand sides of~\eqref{(10.49.32)-rplace} and~\eqref{(10.49.35)-first-replace-INT} gives
\begin{gather*}
1+a\sum_{k=1}^{\infty}\binom{a-1+\frac{k}{2}}{k-1}\frac{(2x)^k}{k}
=1+2xa+2x^2a^2+a\sum_{k=3}^{\infty}\binom{a-1+\frac{k}{2}}{k-1}\frac{(2x)^k}{k}\\
=1+2xa+2x^2a^2+\sum_{k=3}^{\infty} \Biggl[\sum_{m=0}^{k-1} Q(m+1,k-m-1;2) a^{m+1}\Biggr] \frac{(2x)^k}{k!}\\
=1+\Biggl[2x+\sum_{k=3}^{\infty}Q(1,k-1;2) \frac{(2x)^k}{k!}\Biggr]a
+\Biggl[2x^2+\sum_{k=3}^{\infty} Q(2,k-2;2) \frac{(2x)^k}{k!}\Biggr]a^2\\
+\sum_{m=3}^{\infty}\Biggl[\sum_{k=m}^{\infty}Q(m,k-m;2) \frac{(2x)^k}{k!}\Biggr]a^{m}
\end{gather*}
and
\begin{gather*}
\sum_{k=0}^{\infty}\binom{a-1+\frac{k}{2}}{k} \frac{2^kt^{k+1}}{k+1}
=t-\frac{1}{2}t^2+\sum_{k=3}^{\infty}Q(1,k;3)\frac{2^kt^{k+1}}{(k+1)!}\\
+\Biggl[t^2-\frac{2}{3}t^3+\sum_{k=3}^{\infty} Q(2,k-1;3) \frac{2^kt^{k+1}}{(k+1)!}\Biggr]a\\
+\Biggl[\frac{2}{3}t^3+\sum_{k=3}^{\infty}Q(3,k-2;3) \frac{2^kt^{k+1}}{(k+1)!}\Biggr]a^2
+\sum_{m=3}^{\infty}\Biggl[\sum_{k=3}^{\infty}Q(m+1,k-m;3) \frac{2^kt^{k+1}}{(k+1)!}\Biggr]a^m.
\end{gather*}
On the other hand, the left hand sides of~\eqref{(10.49.32)-rplace} and~\eqref{(10.49.35)-first-replace-INT} can be expanded into
\begin{equation*}
\te^{2a\arcsinh x}=\sum_{m=0}^{\infty}\frac{(2\arcsinh x)^m}{m!}a^m
\end{equation*}
and
\begin{align*}
\frac{\exp[(2a-1)\arcsinh t]-1}{2a-1}
&=\sum_{m=1}^{\infty}\frac{(2a-1)^{m-1}(\arcsinh t)^m}{m!}\\
&=\sum_{m=0}^{\infty}\frac{(\arcsinh t)^{m+1}}{(m+1)!} \sum_{q=0}^{m}(-1)^{m-q}\binom{m}{q}(2a)^q\\
&=\sum_{q=0}^{\infty}\Biggl[\sum_{m=q}^{\infty}(-1)^{m-q} \binom{m}{q}\frac{(\arcsinh t)^{m+1}}{(m+1)!}\Biggr](2a)^q\\
&=\sum_{m=0}^{\infty}\Biggl[\sum_{\ell=m}^{\infty}(-1)^{\ell-m}\binom{\ell}{m}\frac{(\arcsinh t)^{\ell+1}}{(\ell+1)!}\Biggr](2a)^m.
\end{align*}
Accordingly, equating coefficients of $a^m$ for $m\in\mathbb{N}_0$, we obtain
\begin{align*}
2\arcsinh x&=2x+\sum_{k=3}^{\infty}Q(1,k-1;2) \frac{(2x)^k}{k!},\\
\frac{(2\arcsinh x)^2}{2!}&=2x^2+\sum_{k=3}^{\infty} Q(2,k-2;2)\frac{(2x)^k}{k!},\\
\frac{(2\arcsinh x)^m}{m!}&=\sum_{k=m}^{\infty}
Q(m,k-m;2) \frac{(2x)^k}{k!},\\
\sum_{\ell=0}^{\infty}(-1)^{\ell}\frac{(\arcsinh t)^{\ell+1}}{(\ell+1)!}
&=t-\frac{t^2}{2}+\sum_{k=3}^{\infty}Q(1,k;3)\frac{2^kt^{k+1}}{(k+1)!}\\
\Biggl[\sum_{\ell=1}^{\infty}(-1)^{\ell-1}\binom{\ell}{1}\frac{(\arcsinh t)^{\ell+1}}{(\ell+1)!}\Biggr]2&=t^2-\frac{2}{3}t^3+\sum_{k=3}^{\infty} Q(2,k-1;3)\frac{2^kt^{k+1}}{(k+1)!},\\
\Biggl[\sum_{\ell=2}^{\infty}(-1)^{\ell-2}\binom{\ell}{2}\frac{(\arcsinh t)^{\ell+1}}{(\ell+1)!}\Biggr]2^2&=\frac{2}{3}t^3+\sum_{k=3}^{\infty}Q(3,k-2;3)\frac{2^kt^{k+1}}{(k+1)!},\\
\Biggl[\sum_{\ell=m}^{\infty}(-1)^{\ell-m}\binom{\ell}{m}\frac{(\arcsinh t)^{\ell+1}}{(\ell+1)!}\Biggr]2^m&=\sum_{k=3}^{\infty}Q(m+1,k-m;3) \frac{2^kt^{k+1}}{(k+1)!}
\end{align*}
for $m\ge3$. Reformulating these seven series expansions and series identities arrives at
\begin{align*}
\frac{\arcsinh x}{x}&=1+\sum_{k=3}^{\infty}Q(1,k-1;2) \frac{(2x)^{k-1}}{k!}\\
&=1+\sum_{k=1}^{\infty}\frac{Q(1,2k;2)}{\binom{2k+1}{1}} \frac{(2x)^{2k}}{(2k)!} +\sum_{k=1}^{\infty}\frac{Q(1,2k+1;2)}{\binom{2k+2}{1}} \frac{(2x)^{2k+1}}{(2k+1)!}\\
&=1+\sum_{k=1}^{\infty}\frac{Q(1,2k;2)}{\binom{2k+1}{1}} \frac{(2x)^{2k}}{(2k)!},\\
\biggl(\frac{\arcsinh x}{x}\biggr)^2&=1+\frac{1}{2x^2}\sum_{k=3}^{\infty} Q(2,k-2;2)\frac{(2x)^k}{k!}\\
&=1+\sum_{k=1}^{\infty}\frac{2Q(2,k;2)}{(k+2)(k+1)} \frac{(2x)^{k}}{k!}\\
&=1+\sum_{k=1}^{\infty}\frac{Q(2,2k;2)}{\binom{2k+2}{2}} \frac{(2x)^{2k}}{(2k)!} +\sum_{k=1}^{\infty}\frac{Q(2,2k-1;2)}{\binom{2k+1}{2}}\frac{(2x)^{2k-1}}{(2k-1)!}\\
&=1+\sum_{k=1}^{\infty}\frac{Q(2,2k;2)}{\binom{2k+2}{2}}\frac{(2x)^{2k}}{(2k)!},\\
\biggl(\frac{\arcsinh x}{x}\biggr)^m&=\sum_{k=0}^{\infty}\frac{Q(m,k;2)}{\binom{m+k}{m}} \frac{(2x)^{k}}{k!}\\
&=1+\sum_{k=1}^{\infty}\frac{1}{\binom{m+2k}{m}} Q(m,2k;2) \frac{(2x)^{2k}}{(2k)!} +\sum_{k=1}^{\infty}\frac{Q(m,2k-1;2)}{\binom{m+2k-1}{m}} \frac{(2x)^{2k-1}}{(2k-1)!}\\
&=1+\sum_{k=1}^{\infty}\frac{1}{\binom{m+2k}{m}} Q(m,2k;2) \frac{(2x)^{2k}}{(2k)!},\\
\sum_{\ell=0}^{\infty}(-1)^{\ell}\frac{(\arcsinh t)^{\ell+1}}{(\ell+1)!}
&=t-\frac{t^2}{2}+\sum_{k=3}^{\infty}Q(1,k;3)\frac{2^kt^{k+1}}{(k+1)!}\\
&=t-\frac{t^2}{2}+\sum_{k=1}^{\infty}Q(1,2k+1;3) \frac{2^{2k+1}t^{2k+2}}{(2k+2)!} +\sum_{k=1}^{\infty}Q(1,2k+2;3) \frac{2^{2k+2}t^{2k+3}}{(2k+3)!},
\end{align*}
and
\begin{align*}
\sum_{\ell=0}^{\infty}(-1)^{\ell}\binom{\ell+1}{1}\frac{(\arcsinh t)^{\ell+2}}{(\ell+2)!}
&=\frac{1}{2}t^2-\frac{1}{3}t^3+\frac{1}{2^2}\sum_{k=3}^{\infty} Q(2,k-1;3)\frac{(2t)^{k+1}}{(k+1)!},\\
\sum_{\ell=0}^{\infty}(-1)^{\ell}\binom{\ell+2}{2}\frac{(\arcsinh t)^{\ell+3}}{(\ell+3)!} &=\frac{1}{6}t^3+\frac{1}{2^3}\sum_{k=3}^{\infty}Q(3,k-2;3)\frac{(2t)^{k+1}}{(k+1)!},\\
\sum_{\ell=0}^{\infty}(-1)^{\ell}\binom{\ell+m}{m}\frac{(\arcsinh t)^{\ell+m+1}}{(\ell+m+1)!} &=\frac{1}{2^{m+1}}\sum_{k=3}^{\infty}Q(m+1,k-m;3) \frac{(2t)^{k+1}}{(k+1)!}
\end{align*}
for $m\ge3$. Consequently, from the first three equations and the last three equations above, we conclude Maclaurin's series expansion~\eqref{arcsinh-series-expansion} and the series identities~\eqref{arcsinh-series-id-1}, \eqref{arcsinh-series-id-2}, and~\eqref{arcsinh-series-id-3}.
\par
In the fourth formula above, by virtue of~\eqref{(12.1)-Quaintance}, we obtain
\begin{equation}\label{Comb-ID-Q(3)-One}
Q(1,2k+1;3)=(2k+1)!\binom{\frac{2k-1}{2}}{2k+1}
\end{equation}
and
\begin{equation}\label{Comb-ID-Q(3)-Two}
Q(1,2k+2;3)=0.
\end{equation}
These two combinatorial identities imply
\begin{equation*}
\sum_{\ell=0}^{\infty}(-1)^{\ell}\frac{(\arcsinh t)^{\ell+1}}{(\ell+1)!}
=t-\frac{t^2}{2}+\sum_{k=1}^{\infty}2^{2k+1}\binom{\frac{2k-1}{2}}{2k+1}\frac{t^{2k+2}}{2k+2}.
\end{equation*}
Furthermore, since
\begin{equation*}
\sum_{\ell=0}^{\infty}(-1)^{\ell}\frac{(\arcsinh t)^{\ell+1}}{(\ell+1)!}
=-\sum_{\ell=1}^{\infty}\frac{(-\arcsinh t)^{\ell}}{\ell!}
=1-\sum_{\ell=0}^{\infty}\frac{(-\arcsinh t)^{\ell}}{\ell!}
=1-\te^{-\arcsinh t},
\end{equation*}
we acquire
\begin{equation*}
1-\te^{-\arcsinh t}
=t-\frac{t^2}{2}+\sum_{k=1}^{\infty}2^{2k+1}\binom{\frac{2k-1}{2}}{2k+1}\frac{t^{2k+2}}{2k+2}.
\end{equation*}
Replacing $t$ by $-t$ in the above equation leads to Maclaurin's series expansion~\eqref{exp-arcsinh-sereis-exapnsion}.
The proof of Theorem~\ref{arcsinh-identity-thm} is complete.
\end{proof}

\begin{cor}\label{arcsinh-idty-cor}
For $m\ge2$, the composite $\Gamma(m,\arcsinh t)$ has Maclaurin's series expansions
\begin{align}\label{arcsinh-series-incomp-1}
\Gamma(2,\arcsinh t)
&=1-\frac{1}{2}t^2+\frac{1}{3}t^3-\frac{1}{4}\sum_{k=3}^{\infty} Q(2,k-1;3)\frac{(2t)^{k+1}}{(k+1)!},\\
\Gamma(3,\arcsinh t)
&=2-\frac{1}{3}t^3-\frac{1}{8}\sum_{k=3}^{\infty}Q(3,k-2;3) \frac{(2t)^{k+1}}{(k+1)!},
\label{arcsinh-series-incomp-2}
\end{align}
and, for $m\ge3$,
\begin{equation}\label{arcsinh-series-incomp-3}
\Gamma(1+m,\arcsinh t) =m!-\frac{m!}{2^{m+1}}\sum_{k=m}^{\infty}Q(m+1,k-m;3) \frac{(2t)^{k+1}}{(k+1)!}.
\end{equation}
where the quantities $Q(m,k;3)$ are given by~\eqref{Q(m-k)-sum-dfn} and the incomplete gamma function $\Gamma(a,x)$ is defined by~\eqref{incomplete-gamma-dfn}.
\end{cor}

\begin{proof}
In~\cite[p.~908, 8.352.2]{Gradshteyn-Ryzhik-Table-8th} and~\cite[Theorem~3]{Jameson-MZ-2016}, there is the formula
\begin{equation*}
\Gamma(1+m,x)=m!\te^{-x}\sum_{k=0}^{m}\frac{x^k}{k!}, \quad m=0,1,2,\dotsc.
\end{equation*}
Hence, it follows that
\begin{align*}
(-1)^m\biggl[1-\frac{\Gamma(1+m,x)}{m!}\biggr]
&=(-1)^m\Biggl(1-\te^{-x}\sum_{k=0}^{m}\frac{x^k}{k!}\Biggr)\\
&=(-1)^m\Biggl[1-\te^{-x}\Biggl(\te^x-\sum_{k=m+1}^{\infty}\frac{x^k}{k!}\Biggr)\Biggr]\\
&=(-1)^m\te^{-x}\sum_{k=m+1}^{\infty}\frac{x^k}{k!}\\
&=(-1)^m\Biggl[\sum_{k=0}^{\infty}(-1)^k\frac{x^k}{k!}\Biggr] \Biggl[\sum_{k=0}^{\infty}\frac{x^{k+m+1}}{(k+m+1)!}\Biggr]\\
&=(-1)^mx^{m+1}\sum_{k=0}^{\infty}\Biggl[\sum_{\ell=0}^{k}\frac{(-1)^\ell}{\ell!}\frac{1}{(k-\ell+m+1)!}\Biggr]x^k\\
&=(-1)^m\sum_{k=0}^{\infty}\Biggl[\sum_{\ell=0}^{k}(-1)^\ell\binom{k+m+1}{\ell}\Biggr]\frac{x^{k+m+1}}{(k+m+1)!}\\
&=(-1)^m\sum_{k=0}^{\infty}\frac{(-1)^k (k+1)}{k+m+1}\binom{k+m+1}{k+1}\frac{x^{k+m+1}}{(k+m+1)!}\\
&=\sum_{k=m}^{\infty}\frac{(-1)^{k} (k-m+1)}{k+1}\binom{k+1}{k-m+1}\frac{x^{k+1}}{(k+1)!}\\
&=\sum_{k=m}^{\infty}(-1)^{k}\binom{k}{m}\frac{x^{k+1}}{(k+1)!}\\
&=\sum_{k=0}^{\infty}(-1)^{k+m}\binom{k+m}{m}\frac{x^{k+m+1}}{(k+m+1)!}
\end{align*}
for $m\in\mathbb{N}_0$, where we used the combinatorial identity
\begin{equation*}
\sum_{\ell=0}^{k}(-1)^\ell\binom{k+m+1}{\ell}=\frac{(-1)^k (k+1)}{k+m+1}\binom{k+m+1}{k+1}
=(-1)^k\binom{k+m}{k}
\end{equation*}
which can be derived from the identity
\begin{equation}\label{Sprugnoli-Gould-2006-p.18(1.5)}
\sum_{k=0}^{n}(-1)^k\binom{x}{k}=(-1)^n\binom{x-1}{n}=\prod_{k=1}^{n}\biggl(1-\frac{x}{k}\biggr)
\end{equation}
in~\cite[p.~18, (1.5)]{Sprugnoli-Gould-2006}.
Substituting this result into the series identities~\eqref{arcsinh-series-id-1}, \eqref{arcsinh-series-id-2}, and~\eqref{arcsinh-series-id-3} in Theorem~\ref{arcsinh-identity-thm} and rearranging yield Maclaurin's series expansions~\eqref{arcsinh-series-incomp-1}, \eqref{arcsinh-series-incomp-2}, and~\eqref{arcsinh-series-incomp-3}.
The proof of Corollary~\ref{arcsinh-idty-cor} is complete.
\end{proof}

\section{Maclaurin's series expansions for positive integer powers of inverse (hyperbolic) tangent function}\label{sec-2guesses}
In this section, we discuss Maclaurin's series expansions of the inverse tangent function $\arctan t$ and the inverse hyperbolic tangent function $\arctanh t$.

\subsection{Maclaurin's series expansion for positive integer powers of inverse tangent function}
It is well known that
\begin{equation*}
\arctan t=\sum_{k=0}^{\infty}(-1)^{k}\frac{t^{2k+1}}{2k+1}, \quad |t|<1.
\end{equation*}
In~\cite[pp.~152--153, (820) and (821)]{Jolley-B-1961}, there are Maclaurin's series expansions
\begin{align*}
\frac{(\arctan t)^2}{2!}&=\sum_{k=0}^{\infty}(-1)^{k}\Biggl(\sum_{\ell=0}^{k}\frac{1}{2\ell+1}\Biggr)\frac{t^{2k+2}}{2k+2}\\
&=\frac{t^2}{2}-\biggl(1+\frac{1}{3}\biggr)\frac{t^4}{4} +\biggl(1+\frac{1}{3}+\frac{1}{5}\biggr)\frac{t^6}{6}-\dotsm
\end{align*}
and
\begin{align*}
\frac{(\arctan t)^3}{3!}&=\sum_{k=0}^{\infty}(-1)^{k}\Biggl(\sum_{\ell_2=0}^{k}\frac{1}{2\ell_2+2}\sum_{\ell_1=0}^{\ell_2} \frac{1}{2\ell_1+1}\Biggr) \frac{t^{2k+3}}{2k+3}\\
&=\frac{1}{2}\frac{t^3}{3} -\biggl[\frac{1}{2}+\frac{1}{4}\biggl(1+\frac{1}{3}\biggr)\biggr]\frac{t^5}{5} +\biggl[\frac{1}{2}+\frac{1}{4}\biggl(1+\frac{1}{3}\biggr)+\frac{1}{6}\biggl(1+\frac{1}{3}+\frac{1}{5}\biggr)\biggr]\frac{t^7}{7}-\dotsm
\end{align*}
for $|t|<1$.
What is the general expression of Maclaurin's series expansion of $(\arctan t)^n$ for $n>3$ and $|t|<1$? We guess that it should be
\begin{align}\notag
\frac{(\arctan t)^n}{n!}&=\sum_{k=0}^{\infty}(-1)^{k} \Biggl(\sum_{\ell_{n-1}=0}^{k}\frac{1}{2\ell_{n-1}+n-1} \sum_{\ell_{n-2}=0}^{\ell_{n-1}}\frac{1}{2\ell_{n-2}+n-2} \dotsm \sum_{\ell_2=0}^{\ell_3}\frac{1}{2\ell_2+2}\sum_{\ell_1=0}^{\ell_2} \frac{1}{2\ell_1+1}\Biggr)\frac{t^{2k+n}}{2k+n}\\
&=\sum_{k=0}^{\infty}(-1)^{k}\Biggl(\prod_{m=1}^{n-1}\sum_{\ell_m=0}^{\ell_{m+1}} \frac{1}{2\ell_{m}+m}\Biggr)\frac{t^{2k+n}}{2k+n}
\label{arctan-power-series-expansion-gen}
\end{align}
for $|t|<1$ and all $n\in\mathbb{N}$ with $\ell_n=k$, where the product is understood to be $1$ if the starting index exceeds the finishing index. For example, when $n=4,5,6$, we have
\begin{align*}
\frac{(\arctan t)^4}{4!}&=\sum_{k=0}^{\infty}(-1)^{k}\Biggl(\sum_{\ell_{3}=0}^{k}\frac{1}{2\ell_{3}+3} \sum_{\ell_2=0}^{\ell_3}\frac{1}{2\ell_2+2}\sum_{\ell_1=0}^{\ell_2} \frac{1}{2\ell_1+1}\Biggr)\frac{t^{2k+4}}{2k+4},\\
\frac{(\arctan t)^5}{5!}&=\sum_{k=0}^{\infty}(-1)^{k}\Biggl(\sum_{\ell_{4}=0}^{k}\frac{1}{2\ell_{4}+4} \sum_{\ell_{3}=0}^{\ell_4}\frac{1}{2\ell_{3}+3} \sum_{\ell_2=0}^{\ell_3}\frac{1}{2\ell_2+2}\sum_{\ell_1=0}^{\ell_2} \frac{1}{2\ell_1+1}\Biggr)\frac{t^{2k+5}}{2k+5},
\end{align*}
and
\begin{equation*}
\frac{(\arctan t)^6}{6!}=\sum_{k=0}^{\infty}(-1)^{k}\Biggl(\sum_{\ell_{5}=0}^{k}\frac{1}{2\ell_{5}+5} \sum_{\ell_{4}=0}^{\ell_5}\frac{1}{2\ell_{4}+4} \sum_{\ell_{3}=0}^{\ell_4}\frac{1}{2\ell_{3}+3} \sum_{\ell_2=0}^{\ell_3}\frac{1}{2\ell_2+2}\sum_{\ell_1=0}^{\ell_2} \frac{1}{2\ell_1+1}\Biggr)\frac{t^{2k+6}}{2k+6}
\end{equation*}
for $|t|<1$.
\par
In~\cite[p.~122, 6.42.3]{Adams-Hippisley-Smithsonian1922}, there is Maclaurin's series expansion
\begin{equation}\label{Schwatt-Phil.Mag.1916-series}
(\arctan x)^p=p!\sum_{k_0=1}^{\infty}(-1)^{k_0-1}\frac{x^{2k_0+p-2}}{2k_0+p-2} \prod_{\alpha=1}^{p-1}\Biggl(\sum_{k_\alpha=1}^{k_{\alpha-1}}\frac{1}{2k_\alpha+p-\alpha-2}\Biggr)
\end{equation}
for $p\in\mathbb{N}$. The Maclaurin's series expansion~\eqref{Schwatt-Phil.Mag.1916-series} was proved in~\cite{Schwatt-Phil.Mag.1916} and is obviously equivalent to~\eqref{arctan-power-series-expansion-gen}. Hence, Maclaurin's series expansion~\eqref{arctan-power-series-expansion-gen} is true.
\par
By the way, in appearance, Maclaurin's series expansion~\eqref{arctan-power-series-expansion-gen} is more beautiful than~\eqref{Schwatt-Phil.Mag.1916-series}.

\subsection{Maclaurin's series expansion for positive integer powers of inverse hyperbolic tangent function}
It is also well known that
\begin{equation*}
\arctanh t=\sum_{k=0}^{\infty}\frac{t^{2k+1}}{2k+1}, \quad |t|<1.
\end{equation*}
Motivated by the difference between~\eqref{arcsin-series-expansion-unify} and~\eqref{arcsinh-series-expansion}, basing on~\eqref{arctan-power-series-expansion-gen}, we further guess that
\begin{equation}
\begin{aligned}\label{arctanh-power-series-expansion-gen}
\frac{(\arctanh t)^n}{n!}&=\sum_{k=0}^{\infty}\Biggl(\sum_{\ell_{n-1}=0}^{k}\frac{1}{2\ell_{n-1}+n-1} \sum_{\ell_{n-2}=0}^{\ell_{n-1}}\frac{1}{2\ell_{n-2}+n-2} \dotsm \sum_{\ell_2=0}^{\ell_3}\frac{1}{2\ell_2+2}\sum_{\ell_1=0}^{\ell_2} \frac{1}{2\ell_1+1}\Biggr)\frac{t^{2k+n}}{2k+n}\\
&=\sum_{k=0}^{\infty}\Biggl(\prod_{m=1}^{n-1}\sum_{\ell_m=0}^{\ell_{m+1}} \frac{1}{2\ell_{m}+m}\Biggr)\frac{t^{2k+n}}{2k+n}
\end{aligned}
\end{equation}
for $|t|<1$ and all $n\in\mathbb{N}$ with $\ell_n=k$, where the product is understood to be $1$ if the starting index exceeds the finishing index. For example, when $n=2,3,4,5,6$, we have
\begin{align*}
\frac{(\arctanh t)^2}{2!}&=\sum_{k=0}^{\infty} \Biggl(\sum_{\ell=0}^{k}\frac{1}{2\ell+1}\Biggr) \frac{t^{2k+2}}{2k+2},\\
\frac{(\arctanh t)^3}{3!}&=\sum_{k=0}^{\infty}\Biggl(\sum_{\ell_2=0}^{k}\frac{1}{2\ell_2+2} \sum_{\ell_1=0}^{\ell_2} \frac{1}{2\ell_1+1}\Biggr) \frac{t^{2k+3}}{2k+3},\\
\frac{(\arctanh t)^4}{4!}&=\sum_{k=0}^{\infty} \Biggl(\sum_{\ell_{3}=0}^{k}\frac{1}{2\ell_{3}+3} \sum_{\ell_2=0}^{\ell_3}\frac{1}{2\ell_2+2}\sum_{\ell_1=0}^{\ell_2} \frac{1}{2\ell_1+1}\Biggr)\frac{t^{2k+4}}{2k+4},\\
\frac{(\arctanh t)^5}{5!}&=\sum_{k=0}^{\infty} \Biggl(\sum_{\ell_{4}=0}^{k}\frac{1}{2\ell_{4}+4} \sum_{\ell_{3}=0}^{\ell_4}\frac{1}{2\ell_{3}+3} \sum_{\ell_2=0}^{\ell_3}\frac{1}{2\ell_2+2}\sum_{\ell_1=0}^{\ell_2} \frac{1}{2\ell_1+1}\Biggr)\frac{t^{2k+5}}{2k+5},
\end{align*}
and
\begin{equation*}
\frac{(\arctanh t)^6}{6!}=\sum_{k=0}^{\infty} \Biggl(\sum_{\ell_{5}=0}^{k}\frac{1}{2\ell_{5}+5} \sum_{\ell_{4}=0}^{\ell_5}\frac{1}{2\ell_{4}+4} \sum_{\ell_{3}=0}^{\ell_4}\frac{1}{2\ell_{3}+3} \sum_{\ell_2=0}^{\ell_3}\frac{1}{2\ell_2+2} \sum_{\ell_1=0}^{\ell_2} \frac{1}{2\ell_1+1}\Biggr)\frac{t^{2k+6}}{2k+6}
\end{equation*}
for $|t|<1$.

\begin{proof}[{Proof of Maclaurin's series expansion~\eqref{arctanh-power-series-expansion-gen}}]
Imitating the proof of~\eqref{Schwatt-Phil.Mag.1916-series} in~\cite{Schwatt-Phil.Mag.1916}, we now give a proof of Maclaurin's series expansion~\eqref{arctanh-power-series-expansion-gen}.
\par
It is clear that
\begin{equation*}
\arctanh t=\int_{0}^{t}\frac{\td x}{1-x^2}
=\sum_{k=0}^{\infty}\int_{0}^{t}x^{2k}\td x
=\sum_{k=0}^{\infty}\frac{t^{2k+1}}{2k+1}
\end{equation*}
and
\begin{align*}
(\arctanh t)^2&=2\int_{0}^{t}\frac{\arctanh x}{1-x^2}\td x\\
&=2\int_{0}^{t}\Biggl(\sum_{k=0}^{\infty}\frac{x^{2k+1}}{2k+1}\Biggr) \Biggl(\sum_{k=0}^{\infty}x^{2k}\Biggr)\td x\\
&=2\int_{0}^{t}\sum_{k=0}^{\infty}\frac{1}{2k+1}\Biggl(\sum_{\ell=0}^{\infty}x^{2\ell+2k+1}\Biggr)\td x\\
&=2\int_{0}^{t}\sum_{k=0}^{\infty}\frac{1}{2k+1}\Biggl(\sum_{\ell=k}^{\infty}x^{2\ell+1}\Biggr)\td x\\
&=2!\sum_{k=0}^{\infty}\frac{1}{2k+1}\Biggl(\sum_{\ell=k}^{\infty}\frac{t^{2\ell+2}}{2\ell+2}\Biggr)\\
&=2!\sum_{\ell=0}^{\infty}\Biggl(\sum_{k=0}^{\ell}\frac{1}{2k+1}\Biggr)\frac{t^{2\ell+2}}{2\ell+2}.
\end{align*}
If Maclaurin's series expansion~\eqref{arctanh-power-series-expansion-gen} is true, then
\begin{align*}
(\arctanh t)^{n+1}&=(n+1)\int_{0}^{t}\frac{(\arctanh x)^{n}}{1-x^2}\td x\\
&=(n+1)!\int_{0}^{t}\sum_{k=0}^{\infty}\Biggl(\prod_{m=1}^{n-1}\sum_{\ell_m=0}^{\ell_{m+1}} \frac{1}{2\ell_{m}+m}\Biggr)\frac{x^{2k+n}}{2k+n}\sum_{\ell=0}^{\infty}x^{2\ell}\td x\\
&=(n+1)!\sum_{k=0}^{\infty}\Biggl(\prod_{m=1}^{n-1}\sum_{\ell_m=0}^{\ell_{m+1}} \frac{1}{2\ell_{m}+m}\Biggr)\frac{1}{2k+n}\sum_{\ell=0}^{\infty}\int_{0}^{t}x^{2\ell+2k+n}\td x\\
&=(n+1)!\sum_{k=0}^{\infty}\Biggl(\prod_{m=1}^{n-1}\sum_{\ell_m=0}^{\ell_{m+1}} \frac{1}{2\ell_{m}+m}\Biggr)\frac{1}{2k+n}\sum_{\ell=0}^{\infty}\frac{x^{2\ell+2k+n+1}}{2\ell+2k+n+1}\\
&=(n+1)!\sum_{k=0}^{\infty}\frac{1}{2k+n}\Biggl(\sum_{\ell_{n-1}=0}^{k}\frac{1}{2\ell_{n-1}+n-1} \prod_{m=1}^{n-2}\sum_{\ell_m=0}^{\ell_{m+1}} \frac{1}{2\ell_{m}+m}\Biggr)\sum_{\ell=k}^{\infty}\frac{x^{2\ell+n+1}}{2\ell+n+1}\\
&=(n+1)!\sum_{\ell=0}^{\infty}\sum_{k=0}^{\ell}\frac{1}{2k+n}\Biggl(\sum_{\ell_{n-1}=0}^{k}\frac{1}{2\ell_{n-1}+n-1} \prod_{m=1}^{n-2}\sum_{\ell_m=0}^{\ell_{m+1}} \frac{1}{2\ell_{m}+m}\Biggr)\frac{x^{2\ell+n+1}}{2\ell+n+1}\\
&=(n+1)!\sum_{k=0}^{\infty}\Biggl(\prod_{m=1}^{n}\sum_{\ell_m=0}^{\ell_{m+1}} \frac{1}{2\ell_{m}+m}\Biggr)\frac{t^{2k+n+1}}{2k+n+1},
\end{align*}
where $\ell_{n+1}=k$.
By induction, Maclaurin's series expansion~\eqref{arctanh-power-series-expansion-gen} is thus proved.
\end{proof}

\section{Useful remarks}\label{sec-power-remarks}

In this section, we state several useful remarks on our main results and related stuffs, including a possibly new combinatorial identity similar to~\eqref{1st-stirling-2k+1} and~\eqref{3rd-comb-id} in Corollary~\ref{arcsin-deriv-comb-id-cor}.

\begin{rem}
Maclaurin's series expansion~\eqref{arcsin-series-expansion-unify} in Theorem~\ref{arcsin-series-expansion-unify-thm} is recovered in~\cite[Section~6]{Taylor-arccos-v2.tex} and is generalized in~\cite[Theorem~4.1]{Maclaurin-series-arccos-v3.tex}.
The closed-form formula~\eqref{Bell-Oertel-closed-Eq} in Theorem~\ref{Bell-Oertel-closed-thm} is reconsidered in~\cite[Theorem~2.2]{Maclaurin-series-arccos-v3.tex}.
\end{rem}

\begin{rem}
In order to avoid the indefinite case $0^0$, we do not include the terms $1$ behind equal signs in~\eqref{arcsin-series-expansion-unify}, \eqref{arcsin-diiff-series}, and~\eqref{arcsinh-series-expansion}, the terms $\frac{1}{2}t^2-\frac{1}{3}t^3$ in~\eqref{arcsinh-series-id-1}, the terms $\frac{1}{3}t^3$ in~\eqref{arcsinh-series-id-2}, the terms $1-\frac{1}{2}t^2+\frac{1}{3}t^3$ in~\eqref{arcsinh-series-incomp-1}, and the terms $2-\frac{1}{3}t^3$ in~\eqref{arcsinh-series-incomp-2} into their corresponding sums, while we do not include the identity~\eqref{1st-stirling-2k+1} into~\eqref{3rd-comb-id}. This idea has been reflected in the proofs of Theorems~\ref{arcsin-series-expansion-unify-thm} and~\ref{arcsinh-identity-thm}.
\end{rem}

\begin{rem}
Theorem~\ref{arcsin-series-expansion-unify-thm}, Theorem~\ref{Bell-Oertel-closed-thm}, and Theorem~\ref{logsine-series-expnsion-thm} give answers to three unification problems posed in~\cite[Remark~5.3]{AIMS-Math20210491.tex}.
\end{rem}

\begin{rem}
When $m=1$, by virtue of~\eqref{(12.1)-Quaintance}, Maclaurin's series expansion~\eqref{arcsin-series-expansion-unify} in Theorem~\ref{arcsin-series-expansion-unify-thm} and Maclaurin's series expansion~\eqref{arcsinh-series-expansion} in Theorem~\ref{arcsinh-identity-thm} become
\begin{equation*}%\label{arcsin-series-expansion-m=1}
\frac{\arcsin t}{t}
=1+\sum_{k=1}^{\infty} (-1)^k\binom{\frac{2k-1}{2}}{2k}\frac{(2t)^{2k}}{2k+1}
\end{equation*}
and
\begin{equation*}%\label{arcsinh-series-expansion-m=1}
\frac{\arcsinh x}{x}
=1+\sum_{k=1}^{\infty} \binom{\frac{2k-1}{2}}{2k} \frac{(2x)^{2k}}{(2k+1)!}.
\end{equation*}
\par
When $k=0$, by virtue of~\eqref{(12.1)-Quaintance}, the series representation~\eqref{logsine-series-expnsion-represent} in Theorem~\ref{logsine-series-expnsion-thm} becomes
\begin{align*}
\ls_j(\theta)
&=2(\ln2)^{j-1}\sin\biggl(\frac{\theta}{2}\biggr) \Biggl[\sum_{q=1}^{\infty} (-1)^{q+1}\binom{\frac{2q-1}{2}}{2q} \biggl(2\sin\frac{\theta}{2}\biggr)^{2q}\\
&\quad\times\sum_{\ell=0}^{j-1}\binom{j-1}{\ell}\biggl(\frac{\ln\sin\frac{\theta}{2}}{\ln2}\biggr)^{\ell} \sum_{p=0}^{\ell}\frac{(-1)^p\langle\ell\rangle_{p}} {(2q+1)^{p+1}\bigl(\ln\sin\frac{\theta}{2}\bigr)^{p}}\\
&\quad-\sum_{\ell=0}^{j-1}\binom{j-1}{\ell} \biggl(\frac{\ln\sin\frac{\theta}{2}}{\ln2}\biggr)^{\ell} \sum_{p=0}^{\ell} \frac{(-1)^p\langle\ell\rangle_{p}}{\bigl(\ln\sin\frac{\theta}{2}\bigr)^{p}}\Biggr],
\end{align*}
where $\ls_j(\theta)=\ls_j^{(0)}(\theta)$ is the logsine function defined by~\eqref{generalized-logsine-dfn} and $\langle\ell\rangle_{p}$ is defined by~\eqref{Fall-Factorial-Dfn-Eq}.
\end{rem}

\begin{rem}
In~\cite[p.~122, 6.42]{Adams-Hippisley-Smithsonian1922}, \cite[pp.~262--263, Proposition~15]{Berndt-Ramanujan-B-I}, \cite[pp.~50--51 and p.~287]{Borwein-Bailey-Girgensohn-Experim-2004}, \cite[p.~384]{Borwein-2-book-87}, \cite[p.~2, (2.1)]{Borwein-Chamberland-IJMMS-2007}, \cite[p.~188, Example~1]{Bromwich-1908}, \cite[Lemma~2]{Chen-CP-ITSF-2012}, \cite[p.~308]{Davydychev-Kalmykov-2001}, \cite[pp.~88--90]{Edwards-1982-DC}, \cite[p.~61, 1.645]{Gradshteyn-Ryzhik-Table-8th}, ~\cite[pp.~124--125, (666); pp.~146--147, (778); pp.~148--149, (783) and (784); pp.~154--155, (832) and~(834); pp.~176--177, (956)]{Jolley-B-1961}, \cite[p.~1011]{Konheim-Wrench-Klamkin-1962Monthly}, \cite[p.~453]{Lehmer-Monthly-1985}, \cite[Section~6.3]{Catalan-Int-Surv.tex}, \cite[p.~126]{Schwatt-B-2nd1924}, \cite{Spiegel-Monthly-1953}, \cite[p.~59, (2.56)]{Wilf-GF-2006-3rd}, or~\cite[p.~676, (2.2)]{Zhang-Chen-JMI-2020}, one can find Maclaurin's series expansions
\begin{equation}
\begin{aligned}\label{Lehmer-Monthly-1985-arcsin-square-expan}
\arcsin x&=\sum_{\ell=0}^{\infty}\frac{1}{2^{2\ell}}\binom{2\ell}{\ell}\frac{x^{2\ell+1}}{2\ell+1},\quad |x|<1,\\
\biggl(\frac{\arcsin x}{x}\biggr)^2&=2!\sum_{k=0}^{\infty} [(2k)!!]^2 \frac{x^{2k}}{(2k+2)!},\quad |x|<1,\\
(\arcsin x)^3&=3!\sum_{\ell=0}^{\infty}[(2\ell+1)!!]^2 \Biggl[\sum_{k=0}^{\ell}\frac{1}{(2k+1)^2}\Biggr] \frac{x^{2\ell+3}}{(2\ell+3)!},\quad |x|<1,
\end{aligned}
\end{equation}
or their variants.
\par
In the paper~\cite{Qi-Chen-Lim-RNA.tex}, those three series expansions in~\eqref{Lehmer-Monthly-1985-arcsin-square-expan} were applied to recover and establish several known and new combinatorial identities containing the ratio of two central binomial coefficients $\binom{2k}{k}$. The central binomial coefficient $\binom{2k}{k}$ is related to the Catalan numbers~\cite{Catalan-Int-Surv.tex} in combinatorial number theory.
In the preprint~\cite{Kobayashi-arXiv-2021}, those three series expansions in~\eqref{Lehmer-Monthly-1985-arcsin-square-expan} and Maclaurin's series expansion of $(\arcsin x)^4$ were applied largely.
\par
Comparing the second series expansion in~\eqref{Lehmer-Monthly-1985-arcsin-square-expan} with Maclaurin's series expansion~\eqref{arcsin-series-expansion-unify} for $m=2$ in Theorem~\ref{arcsin-series-expansion-unify-thm}, we obtain the identity
\begin{equation}\label{2n-stirling-n!square}
Q(2,2k;2)=(-1)^k(k!)^2, \quad k\in\mathbb{N}.
\end{equation}
This combinatorial identity is recovered in~\cite[Lemma~3.1 and Remark~3.3]{Taylor-arccos-v2.tex}.
\par
The combinatorial identity~\eqref{2n-stirling-n!square} and those in~\eqref{Comb-ID-Q(3)-One} and~\eqref{Comb-ID-Q(3)-Two} are possibly new.
\end{rem}

\begin{rem}
By virtue of the formula~\eqref{Sprugnoli-Gould-2006-p.18(1.5)}, we can reformulated the equation~(2.1) in~\cite[Theorem~2.1]{CDM-68111.tex} and the equations~(1.5) and~(1.6) in~\cite[Section~1.3]{Bell-value-elem-funct.tex} as
\begin{equation*}
\bell_{n,k}(\langle\alpha\rangle_1, \langle\alpha\rangle_2, \dotsc,\langle\alpha\rangle_{n-k+1})
=(-1)^k\frac{n!}{k!}\sum_{\ell=0}^{k}(-1)^{\ell}\binom{k}{\ell}\binom{\alpha\ell}{n}
\end{equation*}
and
\begin{equation*}
\bell_{n,k}\Biggl(1, 1-\lambda, (1-\lambda)(1-2\lambda),\dotsc, \prod_{\ell=0}^{n-k}(1-\ell\lambda)\Biggr)
=
\begin{dcases}
(-1)^{k}\frac{\lambda^{n}n!}{k!} \sum_{\ell=0}^k (-1)^{\ell}\binom{k}{\ell}\binom{\ell/\lambda}{n}, & \lambda\ne0\\
S(n,k), & \lambda=0
\end{dcases}
\end{equation*}
for $n\ge k\in\mathbb{N}_0$ and $\alpha,\lambda\in\mathbb{C}$, where $\bell_{n,k}$ is defined by~\eqref{Bell2nd-Dfn-Eq}, the falling factorial $\langle\alpha\rangle_{p}$ is defined by~\eqref{Fall-Factorial-Dfn-Eq}, the second kind Stirling numbers $S(n,k)$ for $n\ge k\in\mathbb{N}_0$ can be analytically generated~\cite[p.~51]{Comtet-Combinatorics-74} by
\begin{equation*}%\label{2stirl-gen-f}
\frac{(\te^x-1)^k}{k!}=\sum_{n=k}^{\infty} S(n,k)\frac{x^n}{n!}
\end{equation*}
and can be explicitly computed~\cite[p.~204, Theorem~A]{Comtet-Combinatorics-74} by
\begin{equation*}%\label{Stirling-Number-dfn}
S(n,k)=
\begin{dcases}
\frac{(-1)^{k}}{k!}\sum_{\ell=0}^k(-1)^{\ell}\binom{k}{\ell}\ell^{n}, & n>k\in\mathbb{N}_0;\\
1, & n=k\in\mathbb{N}_0,
\end{dcases}
\end{equation*}
and extended binomial coefficient $\binom{z}{w}$ is defined by~\eqref{Gen-Coeff-Binom}.
These two identities and those collected in~\cite[Section~1.3 to Section~1.5]{Bell-value-elem-funct.tex} on closed-form formulas for specific values of partial Bell polynomials $\bell_{n,k}$ supply approaches to establish explicit and general formulas of the $m$th derivatives and Maclaurin's series expansions for composite functions $f((a+bx)^\alpha)$, such as $\te^{x^\alpha}$ and $\sin[(a+bx)^\alpha]$, with $\alpha\in\mathbb{R}$, if the $m$th derivatives of the function $f$ can be explicitly or recursively computed for $m\in\mathbb{N}$.
\par
In~\cite[Theorem~1.2]{Tan-Der-App-Thanks.tex}, there are the formulas
\begin{multline}\label{bell-sin-eq}
\bell_{n,k}\biggl(-\sin x,-\cos x,\sin x,\cos x,\dotsc, \cos\biggl[x+\frac{(n-k+1)\pi}{2}\biggr]\biggr)\\
=\frac{(-1)^k\cos^kx}{k!}\sum_{\ell=0}^k\binom{k}{\ell}\frac{(-1)^\ell}{(2\cos x)^\ell}
\sum_{q=0}^\ell\binom{\ell}{q}(2q-\ell)^n \cos\biggl[(2q-\ell)x+\frac{n\pi}2\biggr]
\end{multline}
and
\begin{multline}\label{bell-sin=ans}
\bell_{n,k}\biggl(\cos x,-\sin x,-\cos x,\sin x,\dotsc, \sin\biggl[x+\frac{(n-k+1)\pi}{2}\biggr]\biggr)\\
=\frac{(-1)^k\sin^{k}x}{k!}\sum_{\ell=0}^k\binom{k}{\ell}\frac1{(2\sin x)^{\ell}}
\sum_{q=0}^\ell(-1)^q\binom{\ell}{q}(2q-\ell)^n \cos\biggl[(2q-\ell)x+\frac{(n-\ell)\pi}2\biggr]
\end{multline}
for $n\ge k\in\mathbb{N}$. See also~\cite[Section~1.6]{Bell-value-elem-funct.tex} and closely related references listed therein. These closed-form formulas~\eqref{bell-sin-eq} and~\eqref{bell-sin=ans} provide methods to establish explicit and general formulas of the $m$th derivatives and Maclaurin's series expansions for composite functions $f(\sin x)$ and $f(\cos x)$, such as $\sin^\alpha x$, $\cos^\alpha x$, $\sec^\alpha x$, $\csc^\alpha x$, $\te^{\pm\sin x}$, $\te^{\pm\cos x}$, $\ln\cos x$, $\ln\sin x$, $\ln\sec x$, $\ln\csc x$, $\sin\sin x$, $\cos\sin x$, $\sin\cos x$, $\cos\cos x$, $\tan x$, and $\cot x$ with $\alpha\in\mathbb{R}$, if the $m$th derivatives of the function $f$ can be explicitly or recursively computed for $m\in\mathbb{N}$.
\par
In the paper~\cite{Brychkov-ITSF-2009}, earlier than~\cite{Tan-Der-App-Thanks.tex}, among other things, the $m$th derivatives and Maclaurin's series expansions of the positive integer powers $\sin^nz$, $\cos^nz$, $\tan^nz$, $\cot^nz$, $\sec^nz$, and $\csc^nz$ for $m,n\in\mathbb{N}$ were computed and investigated.
\par
It is not difficult to see that, by virtue of the formulas~\eqref{bell-sin-eq} and~\eqref{bell-sin=ans}, we can deal with explicit and general formulas of the $m$th derivatives and Maclaurin's series expansions of more general functions.
\end{rem}

\begin{rem}
Replacing $\arcsinh t$ by $t$ in Maclaurin's series expansions~\eqref{exp-arcsinh-sereis-exapnsion}, \eqref{arcsinh-series-incomp-1}, \eqref{arcsinh-series-incomp-2}, and~\eqref{arcsinh-series-incomp-3} leads to
\begin{align*}
\te^{t}&=1+\sinh t-(\sinh t)^{2}\sum_{k=0}^{\infty}\binom{\frac{2k-1}{2}}{2k+1}\frac{(2\sinh t)^{2k}}{k+1},\\
\Gamma(2,t)&=1-\frac{1}{2}(\sinh t)^2+\frac{1}{3}(\sinh t)^3-\frac{1}{4}\sum_{k=3}^{\infty} Q(2,k-1;3)\frac{(2\sinh t)^{k+1}}{(k+1)!},\\
\Gamma(3,t)&=2-\frac{1}{3}(\sinh t)^3-\frac{1}{8}\sum_{k=3}^{\infty}Q(3,k-2;3)\frac{(2\sinh t)^{k+1}}{(k+1)!},
\end{align*}
and, for $m\ge3$,
\begin{equation*}
\Gamma(1+m,t) =m!-\frac{m!}{2^{m+1}}\sum_{k=m}^{\infty}Q(m+1,k-m;3) \frac{(2\sinh t)^{k+1}}{(k+1)!},
\end{equation*}
where $Q(m,k;3)$ is defined by~\eqref{Q(m-k)-sum-dfn} and the incomplete gamma function $\Gamma(a,x)$ is defined by~\eqref{incomplete-gamma-dfn}.
\end{rem}

\begin{rem}
In~\cite[pp.~168--169, (901); pp~176--177, (956)]{Jolley-B-1961}, there exist Maclaurin's series expansions
\begin{equation*}
\frac{(\arcsinh\theta)^2}{2!}
=\sum_{k=1}^{\infty}(-1)^{k+1}\frac{[2(k-1)]!!}{(2k-1)!!}\frac{\theta^{2k}}{2k}
=\frac{\theta^2}{2}-\frac{2}{3}\frac{\theta^4}{4}+\frac{2}{3}\frac{4}{5}\frac{\theta^6}{6}-\dotsm
\end{equation*}
and
\begin{equation*}
\frac{\arcsinh\theta}{\sqrt{1+\theta^2}\,}
=\sum_{k=1}^{\infty}(-1)^{k+1}\frac{[2(k-1)]!!}{(2k-1)!!}\theta^{2k-1}
=\theta-\frac{2}{3}\theta^3+\frac{2}{3}\frac{4}{5}\theta^5-\dotsm
\end{equation*}
Comparing the first one with~\eqref{arcsinh-series-expansion} for $m=2$ deduces the identity~\eqref{2n-stirling-n!square} once again.
\end{rem}

\begin{rem}
The series expansion~\eqref{arcsinh-series-expansion} in Theorem~\ref{arcsinh-identity-thm} can be applied to find a closed-form formula for the central factorial numbers of the first kind $t(n,k)$ which can be generated~\cite{centr-bell-polyn.tex} by
\begin{equation*}%\label{t(n-k)-Gen-Eq}
\frac{1}{k!}\Bigl(2\arcsinh\frac{x}{2}\Bigr)^k=\sum_{n=k}^{\infty}t(n,k)\frac{x^n}{n!}, \quad |x|\le2.
\end{equation*}
\end{rem}

\begin{rem}
The Fa\`a di Bruno formula~\cite[Theorem~11.4]{Charalambides-book-2002} and~\cite[p.~139, Theorem~C]{Comtet-Combinatorics-74} can be described in terms of partial Bell polynomials $\bell_{n,k}(x_1,x_2,\dotsc,x_{n-k+1})$ by
\begin{equation}\label{Bruno-Bell-Polynomial}
\frac{\td^n}{\td x^n}f\circ h(x)=\sum_{k=0}^nf^{(k)}(h(x)) \bell_{n,k}\bigl(h'(x),h''(x),\dotsc,h^{(n-k+1)}(x)\bigr)
\end{equation}
for $n\in\mathbb{N}_0$.
It is easy to see that
\begin{equation*}
(\arctan t)^n=\sum_{k=0}^{\infty}\biggl[\lim_{t\to0}\frac{\td^k(\arctan t)^n}{\td t^k}\biggr]\frac{t^k}{k!}
\end{equation*}
and, by employing~\eqref{Bruno-Bell-Polynomial} and considering $u=u(t)=\arctan t\to0$ as $t\to0$,
\begin{align*}
\lim_{t\to0}\frac{\td^k(\arctan t)^n}{\td t^k}
&=\lim_{t\to0}\sum_{\ell=0}^{k}\frac{\td^\ell u^n}{\td u^\ell} \bell_{k,\ell}\Biggl(\frac{1}{1+t^2}, \biggl(\frac{1}{1+t^2}\biggr)',\dotsc, \biggl(\frac{1}{1+t^2}\biggr)^{(k-\ell)}\Biggr)\\
&=\sum_{\ell=0}^{k} \lim_{u\to0} \bigl(\langle n\rangle_\ell u^{n-\ell}\bigr) \lim_{t\to0}\bell_{k,\ell}\Biggl(\frac{1}{1+t^2}, \biggl(\frac{1}{1+t^2}\biggr)',\dotsc, \biggl(\frac{1}{1+t^2}\biggr)^{(k-\ell)}\Biggr)\\
&=n!\bell_{k,n}\Biggl(\frac{1}{1+t^2}\bigg|_{t=0}, \biggl(\frac{1}{1+t^2}\biggr)'\bigg|_{t=0}, \biggl(\frac{1}{1+t^2}\biggr)''\bigg|_{t=0},\dotsc, \biggl(\frac{1}{1+t^2}\biggr)^{(k-n)}\bigg|_{t=0}\Biggr)
\end{align*}
with the convention $\bell_{k,n}=0$ for $n>k$, while, by virtue of~\eqref{Bruno-Bell-Polynomial} and for $\ell\in\mathbb{N}_0$,
\begin{align*}
\biggl(\frac{1}{1+t^2}\biggr)^{(\ell)}\bigg|_{t=0}
&=\sum_{q=0}^{\ell}\frac{\td^q}{\td v^q}\biggl(\frac{1}{1+v}\biggr) \bell_{\ell,q}(2t,2,0,\dotsc,0)\\
&=\lim_{t\to0}\sum_{q=0}^{\ell}\frac{(-1)^qq!}{(1+v)^{q+1}} 2^q\frac{1}{2^{\ell-q}} \frac{\ell!}{q!}\binom{q}{\ell-q}t^{2q-\ell}\\
&=\ell!\lim_{t\to0}\sum_{q=0}^{\ell}(-1)^q \binom{q}{\ell-q}(2t)^{2q-\ell}\\
&=
\begin{dcases}
(-1)^p(2p)!,&\ell=2p\\
0,&\ell=2p+1
\end{dcases}\\
&=\frac{1+(-1)^\ell}{2}(-1)^{\ell/2}\ell!
\end{align*}
for $p\in\mathbb{N}_0$, where we used the substitution $v=v(t)=t^2\to0$ as $t\to0$, the identity
\begin{equation}\label{Bell(n-k)}
\bell_{n,k}\bigl(abx_1,ab^2x_2,\dotsc,ab^{n-k+1}x_{n-k+1}\bigr) =a^kb^n\bell_{n,k}(x_1,x_2,\dotsc,x_{n-k+1})
\end{equation}
for $n\ge k\in\mathbb{N}_0$ and $a,b\in\mathbb{C}$, which can be found in~\cite[p.~412]{Charalambides-book-2002} and~\cite[p.~135]{Comtet-Combinatorics-74}, and the explicit formula
\begin{equation}\label{Bell-x-1-0-eq}
\bell_{n,k}(x,1,0,\dotsc,0)
=\frac{1}{2^{n-k}}\frac{n!}{k!}\binom{k}{n-k}x^{2k-n}
\end{equation}
in~\cite[Theorem~5.1]{Spec-Bell2Euler-S.tex}, \cite[Section~3]{Deriv-Arcs-Cos.tex}, and~\cite[Section~1.4]{Bell-value-elem-funct.tex}, with conventions that $\binom{0}{0}=1$ and $\binom{p}{q}=0$ for $q>p\in\mathbb{N}_0$. Accordingly, we acquire
\begin{align*}
\frac{(\arctan t)^n}{n!}&=\sum_{k=n}^{\infty} \bell_{k,n}\biggl(0!, 0, -2!,0, 4!, 0,-6!\dotsc, \frac{1+(-1)^{k-n}}{2}(-1)^{(k-n)/2}(k-n)!\biggr)\frac{t^k}{k!}\\
&=\sum_{k=0}^{\infty}\bell_{k+n,n}\biggl(0!, 0, -2!,0, 4!, 0,-6!\dotsc, \frac{1+(-1)^{k}}{2}(-1)^{k/2}k!\biggr)\frac{t^{k+n}}{(k+n)!}.
\end{align*}
Comparing this with the guess in~\eqref{arctan-power-series-expansion-gen}, or equivalently with the series expansion~\eqref{Schwatt-Phil.Mag.1916-series}, and equating coefficients of the terms $\frac{t^{k+n}}{k+n}$ yield
\begin{equation}\label{Bell-Arctan-2k}
\bell_{2k+n,n}\bigl(0!, 0, -2!,0, 4!, 0,-6!\dotsc, (-1)^{k}(2k)!\bigr)
=(-1)^{k}\prod_{m=1}^{n-1}\sum_{\ell_m=0}^{\ell_{m+1}} \frac{1}{2\ell_{m}+m}
\end{equation}
and
\begin{equation}\label{Bell-Arctan-2k-1}
\bell_{2k+n-1,n}\bigl(0!, 0, -2!,0, 4!, 0,-6!\dotsc, (-1)^{k-1}(2k-2)!,0\bigr)=0
\end{equation}
for $k,n\in\mathbb{N}$. Since
\begin{align*}
&\quad\bell_{k,n}\bigl((\arctan t)'\big|_{t=0}, (\arctan t)''\big|_{t=0}, (\arctan t)^{(3)}\big|_{t=0},\dotsc, (\arctan t)^{(k-n+1)}\big|_{t=0}\bigr)\\
&=\bell_{k,n}\Biggl(\frac{1}{1+t^2}\bigg|_{t=0}, \biggl(\frac{1}{1+t^2}\biggr)'\bigg|_{t=0}, \biggl(\frac{1}{1+t^2}\biggr)''\bigg|_{t=0},\dotsc, \biggl(\frac{1}{1+t^2}\biggr)^{(k-n)}\bigg|_{t=0}\Biggr)\\
&=\bell_{k,n}\biggl(0!, 0, -2!,0, 4!, 0,-6!\dotsc, \frac{1+(-1)^{k-n}}{2}(-1)^{(k-n)/2}(k-n)!\biggr),
\end{align*}
the identities~\eqref{Bell-Arctan-2k} and~\eqref{Bell-Arctan-2k-1} can be applied to establish Maclaurin's series expansions for composite functions $f(\arctan t)$, if the $m$th derivatives of the function $f$ can be explicitly or recursively computed for $m\in\mathbb{N}$.
\end{rem}

\begin{rem}
It is obvious that
\begin{equation*}
(\arctanh t)^n=\sum_{k=0}^{\infty}\biggl[\lim_{t\to0}\frac{\td^k(\arctanh t)^n}{\td t^k}\biggr]\frac{t^k}{k!}
\end{equation*}
and, by employing~\eqref{Bruno-Bell-Polynomial} and considering $u=u(t)=\arctanh t\to0$ as $t\to0$,
\begin{align*}
\lim_{t\to0}\frac{\td^k(\arctanh t)^n}{\td t^k}
&=\lim_{t\to0}\sum_{\ell=0}^{k}\frac{\td^\ell u^n}{\td u^\ell} \bell_{k,\ell}\Biggl(\frac{1}{1-t^2}, \biggl(\frac{1}{1-t^2}\biggr)',\dotsc, \biggl(\frac{1}{1-t^2}\biggr)^{(k-\ell)}\Biggr)\\
&=\sum_{\ell=0}^{k} \lim_{u\to0} \bigl(\langle n\rangle_\ell u^{n-\ell}\bigr) \lim_{t\to0}\bell_{k,\ell}\Biggl(\frac{1}{1-t^2}, \biggl(\frac{1}{1-t^2}\biggr)',\dotsc, \biggl(\frac{1}{1-t^2}\biggr)^{(k-\ell)}\Biggr)\\
&=n!\bell_{k,n}\Biggl(\frac{1}{1-t^2}\bigg|_{t=0}, \biggl(\frac{1}{1-t^2}\biggr)'\bigg|_{t=0}, \biggl(\frac{1}{1-t^2}\biggr)''\bigg|_{t=0},\dotsc, \biggl(\frac{1}{1-t^2}\biggr)^{(k-n)}\bigg|_{t=0}\Biggr)
\end{align*}
with the convention $\bell_{k,n}=0$ for $n>k$, while, by virtue of~\eqref{Bruno-Bell-Polynomial} and for $\ell\in\mathbb{N}_0$,
\begin{align*}
\biggl(\frac{1}{1-t^2}\biggr)^{(\ell)}\bigg|_{t=0}
&=\sum_{q=0}^{\ell}\frac{\td^q}{\td v^q}\biggl(\frac{1}{1-v}\biggr) \bell_{\ell,q}(2t,2,0,\dotsc,0)\\
&=\lim_{t\to0}\sum_{q=0}^{\ell}\frac{q!}{(1-v)^{q+1}} 2^q\frac{1}{2^{\ell-q}} \frac{\ell!}{q!}\binom{q}{\ell-q}t^{2q-\ell}\\
&=\ell!\lim_{t\to0}\sum_{q=0}^{\ell}\binom{q}{\ell-q}(2t)^{2q-\ell}\\
&=
\begin{dcases}
(2p)!,&\ell=2p\\
0,&\ell=2p+1
\end{dcases}\\
&=\frac{1+(-1)^\ell}{2}\ell!
\end{align*}
for $p\in\mathbb{N}_0$, where we used the substitution $v=v(t)=t^2\to0$ as $t\to0$, the identity~\eqref{Bell(n-k)}, and the explicit formula~\eqref{Bell-x-1-0-eq}. Accordingly, we acquire
\begin{align*}
\frac{(\arctanh t)^n}{n!}&=\sum_{k=n}^{\infty}\bell_{k,n}\biggl(0!, 0, 2!,0, 4!, 0, 6!\dotsc, \frac{1+(-1)^{k-n}}{2}(k-n)!\biggr)\frac{t^k}{k!}\\
&=\sum_{k=0}^{\infty}\bell_{k+n,n}\biggl(0!, 0, 2!,0, 4!, 0,6!\dotsc, \frac{1+(-1)^{k}}{2}k!\biggr)\frac{t^{k+n}}{(k+n)!}.
\end{align*}
Comparing this with the verified guess in~\eqref{arctanh-power-series-expansion-gen} and equating coefficients of the terms $\frac{t^{k+n}}{k+n}$ yield
\begin{equation}\label{Bell-Arctanh-2k}
\bell_{2k+n,n}\bigl(0!, 0,2!,0, 4!, 0,6!\dotsc,(2k)!\bigr)
=\prod_{m=1}^{n-1}\sum_{\ell_m=0}^{\ell_{m+1}} \frac{1}{2\ell_{m}+m}
\end{equation}
and
\begin{equation}\label{Bell-Arctanh-2k-1}
\bell_{2k+n-1,n}\bigl(0!, 0, 2!,0, 4!, 0, 6!\dotsc,(2k-2)!,0\bigr)=0
\end{equation}
for $k,n\in\mathbb{N}$. Since
\begin{align*}
&\quad\bell_{k,n}\bigl((\arctanh t)'\big|_{t=0}, (\arctanh t)''\big|_{t=0}, (\arctanh t)^{(3)}\big|_{t=0},\dotsc, (\arctanh t)^{(k-n+1)}\big|_{t=0}\bigr)\\
&=\bell_{k,n}\Biggl(\frac{1}{1-t^2}\bigg|_{t=0}, \biggl(\frac{1}{1-t^2}\biggr)'\bigg|_{t=0}, \biggl(\frac{1}{1-t^2}\biggr)''\bigg|_{t=0},\dotsc, \biggl(\frac{1}{1-t^2}\biggr)^{(k-n)}\bigg|_{t=0}\Biggr)\\
&=\bell_{k,n}\biggl(0!, 0, 2!,0, 4!, 0,6!\dotsc, \frac{1+(-1)^{k-n}}{2}(k-n)!\biggr),
\end{align*}
the identities~\eqref{Bell-Arctanh-2k} and~\eqref{Bell-Arctanh-2k-1} can be applied to establish Maclaurin's series expansions for composite functions $f(\arctanh t)$, if the $m$th derivatives of the function $f$ can be explicitly or recursively computed for $m\in\mathbb{N}$.
\end{rem}

\begin{rem}
Can one find a simpler expression with less multiplicity of sums for the quantity
\begin{equation*}
\prod_{m=1}^{n-1}\sum_{\ell_m=0}^{\ell_{m+1}} \frac{1}{2\ell_{m}+m}
=\sum_{\ell_{n-1}=0}^{k}\frac{1}{2\ell_{n-1}+n-1} \sum_{\ell_{n-2}=0}^{\ell_{n-1}}\frac{1}{2\ell_{n-2}+n-2} \dotsm\sum_{\ell_{3}=0}^{\ell_4}\frac{1}{2\ell_{3}+3} \sum_{\ell_2=0}^{\ell_3}\frac{1}{2\ell_2+2}\sum_{\ell_1=0}^{\ell_2} \frac{1}{2\ell_1+1}
\end{equation*}
in the brackets of Maclaurin's series expansions~\eqref{arctan-power-series-expansion-gen} and~\eqref{arctanh-power-series-expansion-gen}?
\end{rem}

\begin{rem}
The formula~\eqref{JO(833)} is a variant of
\begin{equation*}
1+\sum_{k=1}^{\infty}\bigl[1+(k-1)^2\bigr]\frac{\sin^k\theta}{k!}
=\frac{\te^\theta}{\cos\theta}
\end{equation*}
and
\begin{equation*}
1+a\theta+\sum_{k=2}^{\infty}a^{\frac{1-(-1)^k}{2}}\bigl[a^2+(k-1)^2\bigr]\frac{\theta^k}{k!}
=\frac{\te^{a\arcsin\theta}}{\sqrt{1-\theta^2}\,}
\end{equation*}
in~\cite[p.~79]{Edwards-B-1899} and~\cite[pp.~118--119, (642); pp.~154--155, (833); pp.~156--157, (839)]{Jolley-B-1961} respectively. For more information, please refer to~\cite[pp.~262--263, Proposition~15]{Berndt-Ramanujan-B-I}, \cite[p.~3]{Borwein-Chamberland-IJMMS-2007}, \cite[p.~308]{Davydychev-Kalmykov-2001}, \cite[Remark~5.3]{AIMS-Math20210491.tex}, and~\cite[pp.~49--50]{Kalmykov-Sheplyakov-lsjk-2005}.
\end{rem}

\begin{rem}
Now we quote some texts in~\cite[pp.~124--125]{Schwatt-B-2nd1924} as follows.
\begin{quote}
Expanding $\sin(tx)$ and $\cos(tx)$ in powers of $\sin x$, we have
\begin{equation*}
\sin(tx)=t\sum_{n=0}^{\infty}(-1)^n\prod_{k=1}^{n}\bigl[t^2-(2k-1)^2\bigr]\frac{\sin^{2n+1}x}{(2n+1)!}
\end{equation*}
and
\begin{equation*}
\cos(tx)=\sum_{n=0}^{\infty}(-1)^n\prod_{k=0}^{n-1}\bigl[t^2-(2k)^2\bigr]\frac{\sin^{2n}x}{(2n)!}
\end{equation*}
for $|x|<\frac{\pi}{2}$ and all values of $t$. But
\begin{equation*}
\sin(tx)=\sum_{n=0}^{\infty}(-1)^n\frac{(tx)^{2n+1}}{(2n+1)!}
\end{equation*}
and
\begin{equation*}
\cos(tx)=\sum_{n=0}^{\infty}(-1)^n\frac{(tx)^{2n}}{(2n)!}.
\end{equation*}
\end{quote}
These texts recited from~\cite[pp.~124--125]{Schwatt-B-2nd1924} are equivalent to the equality
\begin{equation}\label{exp-arsin-b(t)0series}
\te^{t\arcsin x}=\sum_{\ell=0}^{\infty}\frac{b_\ell(t)x^\ell}{\ell!}
\end{equation}
used in~\cite[pp.~262--263, Proposition~15]{Berndt-Ramanujan-B-I}, \cite[p.~3]{Borwein-Chamberland-IJMMS-2007}, \cite[p.~308]{Davydychev-Kalmykov-2001}, and~\cite[pp.~49--50]{Kalmykov-Sheplyakov-lsjk-2005}, where $b_0(t)=1$, $b_1(t)=t$, and
\begin{equation*}
b_{2\ell}(t)=\prod_{k=0}^{\ell-1}\bigl[t^2+(2k)^2\bigr], \quad b_{2\ell+1}(t)=t\prod_{k=1}^{\ell}\bigl[t^2+(2k-1)^2\bigr]
\end{equation*}
for $\ell\in\mathbb{N}$. The equality~\eqref{exp-arsin-b(t)0series} has also been applied in Section~2 in the paper~\cite{AIMS-Math20210491.tex}.
\par
In~\cite[Lemmas~3.1 and~3.2]{Taylor-arccos-v2.tex}, the quantities $b_{2\ell}(t)$ and $b_{2\ell+1}(t)$ are expanded as finite sums in terms of the first kind Stirling numbers $s(n,k)$.
\end{rem}

\begin{rem}
In~\cite[p.~124]{Schwatt-B-2nd1924}, there are Maclaurin's series expansions
\begin{equation*}
(\arctan x)^p=\sum_{n=0}^{\infty}(-1)^nx^{2n+p} \prod_{k=1}^{p-1} \Biggl(\sum_{n_k=0}^{n_{k-1}} \frac{1}{2n_{k-1}-2n_k+1}\Biggr) \frac{1}{2n_{p-1}+1}
\end{equation*}
and
\begin{align*}
(\arcsin x)^p&=\sum_{n=0}^{\infty}x^{2n+p} \prod_{k=1}^{p-1}\Biggl[\sum_{n_k=0}^{n_{k-1}} \frac{1}{2^{2(n_{k-1}-n_k)}(2n_{k-1}-2n_k+1)}\\
&\quad\times\binom{2n_{k-1}-n_k}{n_{k-1}-n_k} \frac{1}{2^{2n_{p-1}}(2n_{p-1}+1)}\binom{2n_{p-1}}{n_{p-1}}\Biggr],
\end{align*}
where $n_0=n$ and $p\in\mathbb{N}$. These two Maclaurin's series expansions are not more beautiful than~\eqref{arctan-power-series-expansion-gen} and~\eqref{arcsin-series-expansion-unify} respectively.
\par
In~\cite[p.~128]{Schwatt-B-2nd1924}, there is Maclaurin's series expansion
\begin{align*}
(\arcsec x)^p&=(-1)^p\sum_{n=0}^{\infty}\frac{1}{x^{2n+p}}\prod_{k=1}^{p-1}\Biggl[\sum_{n_k=0}^{n_{k-1}} \frac{1}{2^{2(n_{k-1}-n_k)} (2n_{k-1}-2n_k+1)}\\
&\quad\times\binom{2n_{k-1}-n_k}{n_{k-1}-n_k} \frac{1}{2^{2n_{p-1}}(2n_{p-1}+1)}\binom{2n_{p-1}}{n_{p-1}}\Biggr],
\end{align*}
where $n_0=n$ and $p\in\mathbb{N}$.
\end{rem}

\begin{rem}
The identity~\eqref{1st-stirling-2k+1} is a special of the known identity~\eqref{(12.1)-Quaintance}.
\par
The identities~\eqref{1st-stirling-2k+1} and~\eqref{3rd-comb-id} in Corollary~\ref{arcsin-deriv-comb-id-cor} are also proved in the proof of Theorem~\ref{arcsinh-identity-thm}.
\par
The quantity $Q(m,k;\alpha)$ in~\eqref{Q(m-k)-sum-dfn} can be equivalently reformulated as
\begin{equation}\label{sum-(12.1)-Quaintance}
\sum _{\ell=0}^k \binom{m+\ell}{m} s(m+k,m+\ell)z^\ell, \quad k,m\in\mathbb{N}_0, \quad z\ne0.
\end{equation}
The finite sum on the right hand side of the identity~\eqref{(12.1)-Quaintance} is a special case $m=0$ of the finite sum in~\eqref{sum-(12.1)-Quaintance}.
\par
When replacing $k$ by $2k-1$ and taking $z=\frac{2k+m-2}{2}$ in~\eqref{sum-(12.1)-Quaintance}, we derive the finite sum on the right hand side of the identity~\eqref{3rd-comb-id}.
The quantity in the bracket on the right hand side of Maclaurin's series expansion~\eqref{arcsin-series-expansion-unify} is also a special case of the finite sum~\eqref{sum-(12.1)-Quaintance}. The identity~\eqref{2n-stirling-n!square} gives a sum of~\eqref{sum-(12.1)-Quaintance} for taking $m=1$ and $z=k$ and for replacing $k$ by $2k$.
\par
Does there exist a simpler and general expression for the sum~\eqref{sum-(12.1)-Quaintance}, or say, for the quantity $Q(m,k;\alpha)$ in~\eqref{Q(m-k)-sum-dfn}? If yes, Maclaurin's series expansions~\eqref{arcsin-series-expansion-unify} and~\eqref{arcsin-diiff-series} in Theorem~\ref{arcsin-series-expansion-unify-thm} and Corollary~\ref{arcsin-diiff-series-cor}, the closed-form formula~\eqref{Bell-Oertel-closed-Eq} in Theorem~\ref{Bell-Oertel-closed-thm}, the series representation~\eqref{logsine-series-expnsion-represent} in Theorem~\ref{logsine-series-expnsion-thm}, Maclaurin's series expansions~\eqref{arcsinh-series-expansion} in Theorem~\ref{arcsinh-identity-thm}, the series identities~\eqref{arcsinh-series-id-1}, \eqref{arcsinh-series-id-2}, and~\eqref{arcsinh-series-id-3} in Theorem~\ref{arcsinh-identity-thm}, and Maclaurin's series expansions~\eqref{arcsinh-series-incomp-1}, \eqref{arcsinh-series-incomp-2}, and~\eqref{arcsinh-series-incomp-3} in Corollary~\ref{arcsinh-idty-cor} would be further simplified.
\end{rem}

\begin{rem}
It is common knowledge that
\begin{equation*}
\arcsin t+\arccos t=\frac{\pi}{2}, \quad |t|<1.
\end{equation*}
This means that
\begin{align*}
(\arccos t)^m&=\biggl(\frac{\pi}{2}-\arcsin t\biggr)^m\\
&=\sum_{q=0}^{m}(-1)^q\binom{m}{q}\biggl(\frac{\pi}{2}\biggr)^{m-q}(\arcsin t)^q\\
&=\biggl(\frac{\pi}{2}\biggr)^{m}+\sum_{q=1}^{m}(-1)^q\binom{m}{q}\biggl(\frac{\pi}{2}\biggr)^{m-q} t^q\Biggl[1+\sum_{k=1}^{\infty} \frac{(-1)^k}{\binom{q+2k}{q}} Q(q,2k;2) \frac{(2t)^{2k}}{(2k)!}\Biggr]\\
&=\biggl(\frac{\pi}{2}\biggr)^{m}+\sum_{q=1}^{m}(-1)^q\binom{m}{q}\biggl(\frac{\pi}{2}\biggr)^{m-q} t^q\\
&\quad+\sum_{q=1}^{m}\sum_{k=1}^{\infty} \binom{m}{q}\biggl(\frac{\pi}{2}\biggr)^{m-q}\frac{(-1)^{q+k}}{\binom{q+2k}{q}} Q(q,2k;2) \frac{2^{2k}t^{q+2k}}{(2k)!}\\
&=\biggl(\frac{\pi}{2}\biggr)^{m}+\sum_{q=1}^{m}(-1)^q\binom{m}{q}\biggl(\frac{\pi}{2}\biggr)^{m-q} t^q\\
&\quad+\sum_{k=1}^{\infty} \sum_{q=1}^{m}\binom{m}{q}\biggl(\frac{\pi}{2}\biggr)^{m-q}\frac{(-1)^{q+k}}{\binom{q+2k}{q}}Q(q,2k;2) \frac{2^{2k}t^{q+2k}}{(2k)!}\\
&=\biggl(\frac{\pi}{2}\biggr)^{m}+\sum_{q=1}^{m}(-1)^q\binom{m}{q}\biggl(\frac{\pi}{2}\biggr)^{m-q} t^q\\
&\quad+\sum_{k=1}^{\infty}(-4)^k \sum_{q=1}^{m} (-1)^{q}q!\binom{m}{q} \biggl(\frac{\pi}{2}\biggr)^{m-q} Q(q,2k;2) \frac{t^{q+2k}}{(q+2k)!}\\
&=\biggl(\frac{\pi}{2}\biggr)^{m}+\sum_{q=1}^{m}(-1)^q\binom{m}{q}\biggl(\frac{\pi}{2}\biggr)^{m-q} t^q\\
&\quad+\sum_{p=3}^{\infty} \Biggl[\sum_{q+2k=p}^{q,k\in\mathbb{N}}(-4)^k(-1)^{q}q!\binom{m}{q} \biggl(\frac{\pi}{2}\biggr)^{m-q}Q(q,p-q;2)\Biggr]\frac{t^{p}}{p!}
\end{align*}
for $|t|<1$, that is,
\begin{equation}
\begin{aligned}\label{arccos-ser-exp-ugly}
(\arccos t)^m&=\biggl(\frac{\pi}{2}\biggr)^{m}+\sum_{q=1}^{m}(-1)^q\binom{m}{q}\biggl(\frac{\pi}{2}\biggr)^{m-q} t^q\\
&\quad+\sum_{p=3}^{\infty} \Biggl[\sum_{k=1}^{\infty} (-4)^k\sum_{q=1}^{p-2k}(-1)^{q}q!\binom{m}{q} \biggl(\frac{\pi}{2}\biggr)^{m-q}Q(q,p-q;2)\Biggr]\frac{t^{p}}{p!}
\end{aligned}
\end{equation}
for $|t|<1$, where we used the power series expansion~\eqref{arcsin-series-expansion-unify} in Theorem~\ref{arcsin-series-expansion-unify}, used the convention $\binom{u}{v}=0$ for $u<v$, and understood the sum, if the starting index exceeds the finishing index, to be zero.
\par
Substituting the relation $\arccos t=-\ti\arccosh t$ into~\eqref{arccos-ser-exp-ugly}, we can derive Maclaurin's series expansion of the inverse hyperbolic cosine $\arccosh t$.
\par
In the papers~\cite{Maclaurin-series-arccos-v3.tex, Taylor-arccos-v2.tex}, we will further discover nicer and more beautiful Maclaurin's and Taylor's series expansions of functions related to $(\arccos t)^\alpha$ and $(\arcsin t)^\alpha$ for $\alpha\in\mathbb{R}$.
\end{rem}

\begin{rem}
Since the relation
\begin{equation}\label{arcsinh=arcsin-rel}
\arcsinh t=-\ti\arcsin(\ti t)
\end{equation}
or, equivalently,
\begin{equation}\label{arcsinh=arcsin-rel-equiv}
\arcsin t=-\ti\arcsinh(\ti t),
\end{equation}
the series expansions~\eqref{arcsin-series-expansion-unify} and~\eqref{arcsinh-series-expansion} are equivalent to each other.
\par
Applying the relation~\eqref{arcsinh=arcsin-rel-equiv} into the series expansion~\eqref{arcsin-diiff-series} in Corollary~\ref{arcsin-diiff-series-cor} results in the series expansion
\begin{equation*}%\label{arcsinh-diiff-series}
\frac{(\arcsinh t)^{m}}{\sqrt{1+t^2}\,}
=t^m\Biggl[1+\sum_{k=1}^{\infty} \frac{1}{\binom{m+2k}{m}} Q(m+1,2k;2) \frac{(2t)^{2k}}{(2k)!}\Biggr]
\end{equation*}
for $m\in\mathbb{N}_0$ and $|t|<1$, where $Q(m+1,2k;2)$ is defined by~\eqref{Q(m-k)-sum-dfn}.
\par
Considering the relation~\eqref{arcsinh=arcsin-rel} in~\eqref{arcsinh-series-id-1}, \eqref{arcsinh-series-id-2}, \eqref{arcsinh-series-id-3}, and~\eqref{exp-arcsinh-sereis-exapnsion} in Theorem~\ref{arcsinh-identity-thm} deduces
\begin{align*}
\sum_{\ell=0}^{\infty}(-1)^{\ell}(\ell+1)\frac{[-\ti\arcsin(\ti t)]^{\ell+2}}{(\ell+2)!}
&=\frac{1}{2}t^2-\frac{1}{3}t^3+\frac{1}{4}\sum_{k=3}^{\infty} Q(2,k-1;3)\frac{(2t)^{k+1}}{(k+1)!},\\
\sum_{\ell=0}^{\infty}(-1)^{\ell}(\ell+1)(\ell+2)\frac{[-\ti\arcsin(\ti t)]^{\ell+3}}{(\ell+3)!} &=\frac{1}{3}t^3+\frac{1}{8}\sum_{k=3}^{\infty}Q(3,k-2;3)\frac{(2t)^{k+1}}{(k+1)!},\\
\sum_{\ell=0}^{\infty}(-1)^{\ell}\binom{\ell+m}{m}\frac{[-\ti\arcsin(\ti t)]^{\ell+m+1}}{(\ell+m+1)!} &=\frac{1}{2^{m+1}}\sum_{k=m}^{\infty}Q(m+1,k-m;3) \frac{(2t)^{k+1}}{(k+1)!}
\end{align*}
for $m\ge3$, and
\begin{equation*}
\te^{-\ti\arcsin(\ti t)}=1+t-t^2\sum_{k=0}^{\infty}\binom{\frac{2k-1}{2}}{2k+1}\frac{(2t)^{2k}}{k+1},
\end{equation*}
where $Q(m+1,k-m;3)$ is defined by~\eqref{Q(m-k)-sum-dfn}.
\par
Considering the relation~\eqref{arcsinh=arcsin-rel} in~\eqref{arcsinh-series-incomp-1}, \eqref{arcsinh-series-incomp-2}, and~\eqref{arcsinh-series-incomp-3} in Corollary~\ref{arcsinh-idty-cor} deduces
\begin{align*}
\Gamma(2,-\ti\arcsin(\ti t))
&=1-\frac{1}{2}t^2+\frac{1}{3}t^3-\frac{1}{4}\sum_{k=3}^{\infty} Q(2,k-1;3)\frac{(2t)^{k+1}}{(k+1)!},\\
\Gamma(3,-\ti\arcsin(\ti t))
&=2-\frac{1}{3}t^3-\frac{1}{8}\sum_{k=3}^{\infty}Q(3,k-2;3)\frac{(2t)^{k+1}}{(k+1)!},
\end{align*}
and, for $m\ge3$,
\begin{equation*}
\Gamma(1+m,-\ti\arcsin(\ti t)) =m!-\frac{m!}{2^{m+1}}\sum_{k=m}^{\infty}Q(m+1,k-m;3) \frac{(2t)^{k+1}}{(k+1)!}.
\end{equation*}
where $Q(m+1,k-m;3)$ is given by~\eqref{Q(m-k)-sum-dfn} and the incomplete gamma function $\Gamma(a,x)$ is given by~\eqref{incomplete-gamma-dfn}.
\end{rem}

\begin{rem}
All of Maclaurin's series expansions of positive integer powers of the inverse (hyperbolic) trigonometric functions in this paper can be used to derive infinite series representations of positive integer powers of the circular constant $\pi$. For example, taking $t=\frac{1}{2}$ in~\eqref{arcsin-series-expansion-unify} and simplifying result in
\begin{equation*}
\biggl(\frac{\pi}{3}\biggr)^{m}
=1+m!\sum_{k=1}^{\infty} (-1)^k\frac{Q(m,2k;2)}{(m+2k)!}, \quad m\in\mathbb{N}.
\end{equation*}
\end{rem}

\begin{rem}
This paper is a revised version of the preprint~\cite{Ser-Pow-Arcs-Arctan-Simp.tex}, a continuation of the paper~\cite{AIMS-Math20210491.tex}, and a companion of the articles~\cite{Maclaurin-series-arccos-v3.tex, Taylor-arccos-v2.tex, Wilf-Ward-2010P.tex}.
\end{rem}

\section{Declarations}

\begin{description}
\item[Acknowledgements]
The authors thank
\begin{enumerate}
\item
Mr. Chao-Ping Chen (Henan Polytechnic University, China; chenchaoping@sohu.com) for his asking the combinatorial identity in~\cite[Theorem~2.1]{Qi-Chen-Lim-RNA.tex} via Tencent QQ on 18 December 2020. Since then, we communicated and discussed with each other many times.
\item
Mr. Mikhail Yu. Kalmykov (Bogoliubov Laboratory of Theoretical Physics, Joint Institute for Nuclear Research, Russia; kalmykov.mikhail@googlemail.com) for his providing the references~\cite{Davydychev-Kalmykov-2001, Kalmykov-Sheplyakov-lsjk-2005, Lewin-B-1981} on 9 and 27 January 2021. We communicated and discussed with each other many times.
\item
Mr. Frank Oertel (Philosophy, Logic \& Scientific Method Centre for Philosophy of Natural and Social Sciences, London School of Economics and Political Science, UK; f.oertel@email.de) for his citing the paper~\cite{Bell-value-elem-funct.tex} and sending the paper~\cite{Borwein-Chamberland-IJMMS-2007} to the authors and others. We communicated and discussed with each other many times.
\item
Mr. Fr\'ed\'eric Ouimet (California Institute of Technology, USA; McGill Univrsity, Canada; ouimetfr@caltech.edu, frederic.ouimet2@mcgill.ca) for his photocopying by Caltech Library Services and transferring via ResearchGate those two pages containing the formula~\eqref{arcsin-pochhammer} on 2 February 2021.
\item
Mr. Christophe Vignat (Department of Physics, Universite d'Orsay, France; Department of Mathematics, Tulane University, USA; cvignat@tulane.edu) for his sending electronic version of those pages containing the formulas~\eqref{arcsin-pochhammer} and~\eqref{extended-Pochhammer-dfn} in~\cite{Hansen-B-1975, Schwatt-B-2nd1924} on 30 January 2021 and for his sending electronic version of the monograph~\cite{Jolley-B-1961} on 8 February 2021.
\item
Mr. Fei Wang (Department of Mathematics, Zhejiang Institute of Mechanical and Electrical Engineering, China; wf509529@163.com) for his frequent communications and helpful discussions with the authors via Tencent QQ online.
\item
Mr. Li Yin (Binzhou University, China; yinli7979@163.com) for his frequent communications and helpful discussions with the authors via Tencent QQ online.
\end{enumerate}

\item[Availability of data and material]
Data sharing is not applicable to this article as no new data were created or analyzed in this study.

\item[Competing interests]
The authors declare that they have no conflict of competing interests.

\item[Funding]
The author Dongkyu Lim was partially supported by the National Research Foundation of Korea under Grant NRF-2021R1C1C1010902, Republic of Korea.

\item[Authors' contributions]
All authors contributed equally to the manuscript and read and approved the final manuscript.
\end{description}

\end{document}